\documentclass[12pt,a4paper]{amsart}

\usepackage{amssymb}
\usepackage{amscd}
\usepackage{hyperref}

\def\frak{\mathfrak}
\def\Bbb{\mathbb}
\def\Cal{\mathcal}

\let\phi\varphi

\newcommand{\x}{\times}
\renewcommand{\o}{\circ}

\newcommand{\be}{\beta}

\newcommand{\de}{\delta}
\newcommand{\ep}{\epsilon}

\newcommand{\om}{\omega}
\newcommand{\ph}{\phi}
\newcommand{\ps}{\psi}
\renewcommand{\th}{\theta}
\newcommand{\si}{\sigma}
\newcommand{\ze}{\zeta}
\newcommand{\Ga}{\Gamma}
\newcommand{\La}{\Lambda}
\newcommand{\Ph}{\Phi}
\newcommand{\Ps}{\Psi}
\newcommand{\Om}{\Omega}

\newcommand{\Up}{\Upsilon}
\def\Rho{\mbox{\textsf{P}}}

\newcommand{\barm}{\overline{M}}
\newcommand{\vol}{\operatorname{vol}}
\newcommand{\im}{\operatorname{im}}
\newcommand{\id}{\operatorname{id}}
\newcommand{\tfp}{\operatorname{tfp}}
\newcommand{\Ric}{\operatorname{Ric}}
\newcommand{\Herm}{\operatorname{Herm}}
\newcommand{\End}{\operatorname{End}}

\newcommand{\rpl}                         
{\mbox{$
\begin{picture}(12.7,8)(-.5,-1)
\put(0,0.2){$+$}
\put(4.2,2.8){\oval(8,8)[r]}
\end{picture}$}}

\numberwithin{equation}{section}

\newcounter{theorem}
\newtheorem{thm}[theorem]{Theorem}
\newtheorem*{thm*}{Theorem \thesubsection}
\newtheorem{lemma}[theorem]{Lemma}
\newtheorem{prop}[theorem]{Proposition}
\newtheorem{cor}[theorem]{Corollary}
\newtheorem*{lemma*}{Lemma \thesubsection}
\newtheorem*{prop*}{Proposition \thesubsection}
\newtheorem*{cor*}{Corollary \thesubsection}

\theoremstyle{definition}
\newtheorem{definition}[theorem]{Definition}
\newtheorem*{definition*}{Definition \thesubsection}

\newtheorem*{example*}{Example \thesubsection}
\theoremstyle{remark}
\newtheorem{remark}[theorem]{Remark}
\newtheorem*{remark*}{Remark \thesubsection}

\def\sideremark#1{\ifvmode\leavevmode\fi\vadjust{\vbox to0pt{\vss
 \hbox to 0pt{\hskip\hsize\hskip1em
 \vbox{\hsize3cm\tiny\raggedright\pretolerance10000
  \noindent #1\hfill}\hss}\vbox to8pt{\vfil}\vss}}}%
                        
\begin{document}
\renewcommand{\today}{}
\title{C-Projective Compactification;\\
 (quasi--)K\"ahler Metrics and CR boundaries}

\author{Andreas \v Cap and A.\ Rod Gover}

\address{A.\v C.: Faculty of Mathematics\\
University of Vienna\\
Oskar--Morgenstern--Platz 1\\
1090 Wien\\
Austria\\
A.R.G.:Department of Mathematics\\
  The University of Auckland\\
  Private Bag 92019\\
  Auckland 1142\\
  New Zealand;\\
Mathematical Sciences Institute\\
Australian National University \\ ACT 0200, Australia} 
\email{Andreas.Cap@univie.ac.at}
\email{r.gover@auckland.ac.nz}

\begin{abstract}
  For complete complex connections on almost complex manifolds we
  introduce a natural definition of compactification. This is based on
  almost c--projective geometry, which is the almost complex analogue
  of projective differential geometry. The boundary at infinity is a
  (possibly non-integrable) CR structure.  The theory applies to
  almost Hermitean manifolds which admit a complex metric connection
  of minimal torsion, which means that they are quasi--K\"ahler in the
  sense of Gray--Hervella; in particular it applies to K\"ahler and nearly
  K\"{a}hler manifolds. Via this canonical connection, we obtain a notion of
  c-projective compactification for quasi--K\"ahler metrics of any
  signature.

  We describe an asymptotic form for metrics that is necessary and
  sufficient for c--projective compactness. This metric form provides
  local examples and, in particular, shows that the usual complete
  K\"{a}hler metrics associated to smoothly bounded, strictly
  pseudoconvex domains in $\mathbb{C}^n$ are c--projectively
  compact. For a smooth manifold with boundary and a complete
  quasi-K\"ahler metric $g$ on the interior, we show that if its
  almost c--projective structure extends smoothly to the boundary then
  so does its scalar curvature. We prove that $g$ is almost
  c--projectively compact if and only if this scalar curvature is
  non-zero on an open dense set of the boundary, in which case it is,
  along the boundary, locally constant and hence nowhere zero there.
  Finally we describe the asymptotics of the curvature, showing, in
  particular, that the canonical connection satisfies an asymptotic
  Einstein condition. Key to much of the development is a certain real
  tractor calculus for almost c--projective geometry, and this is
  developed in the article.
\end{abstract}

\maketitle



\pagestyle{myheadings} \markboth{\v Cap, Gover}{C--Projective
  compactness}

\thanks{Both authors gratefully acknowledge support from the Royal
  Society of New Zealand via Marsden Grant 10-UOA-113; A\v C
  gratefully acknowledges support by project P27072-N25 of the
  Austrian Science Fund (FWF) and also the hospitality of the
  University of Auckland. }

\section{Introduction}\label{1}

Consider a smooth manifold $\barm$ with boundary $\partial M$,
interior $M$, and a geometric structure on $M$ which does not admit a
smooth extension to the boundary; for example a complete Riemannian
metric. It is then natural to ask whether some aspect or weakening of
the interior geometry admits a smooth extension to the boundary and
gives rise to a geometric structure there. If this works, one can
start to relate asymptotics of the interior geometry to the boundary
geometry on various levels (geometric objects, equations of geometric
origin, and so forth) and also deduce consequences for the spectral
theory of suitable operators. Such ideas are fundamental for several
areas of mathematics and theoretical physics, including GR, scattering
theory, complex analysis, conformal and parabolic geometries,
representation theory, and holography see
e.g.~\cite{AdSCFTreview,AnkOrst,Feff,Fr,GrWitt,GSB,HislopPerry,MazzMel,Melrose}.

The most well known example of this idea is conformal
compactification; a complete metric on $M$ is conformally compact if
its conformal structure extends to the boundary $\partial M$ and it
has a suitably uniform asymptotic volume growth toward $\partial
M$. In this case, the boundary inherits a conformal structure and the
resulting structure provides a deep link between this conformal
geometry and the interior (pseudo--)Riemannian geometry.  

Guided by examples arising from reductions of projective holonomy
\cite{ageom,hol-red}, a notion of compactification linked to
projective differential geometry was introduced in
\cite{proj-comp,proj-comp2}. This concept is initially defined for
affine connections, so, via the Levi--Civita connection, it
automatically applies to pseudo--Riemannian metrics. From the point of
view of geometric analysis this {\em projective compactification} is
motivated by its strong links to geodesic structure and the fact that
many of the natural equations studied in Riemannian geometry and
physics have projective invariance, but not conformal invariance. An
implicit use of projective compactification is important in recent
microlocal analysis advances of Vasy in \cite{Vforms,Vasy}.

In this article, we introduce a concept of compactification that is
suitable for complex geometries (and also the more general almost
complex geometries); it may be viewed as an analogue of the projective
compactification of order two from \cite{proj-comp,proj-comp2},
although the complex situation is considerably more subtle. Hence we
make contact with complex analysis, where the study of domains via the
geometry of their boundaries is of fundamental importance
\cite{Feff,Cheng-Yau,Hirachi}.  The appropriate almost complex version of
projective geometry has been classically studied under the name
h--projective (or holomorphically projective) geometry. This name is
misleading however, since it is not holomorphic in nature. Recently
there has been renewed interest in these geometries, see e.g.\
\cite{MR1,MR2}, also because they provide an example of a parabolic
geometry.  Adopting the latter point of view, a general basic theory
of these geometries has been worked out under the name \textit{almost
  c--projective structures} in \cite{CEMN}, which will be one of our
standard references.

An almost c--projective structure can be defined on any almost complex manifold,
and is given by an equivalence class of linear connections on the
tangent bundle which preserve the almost complex structure and have
minimal torsion. Hence in the situation of a manifold $\barm$ with
boundary $\partial M$ and interior $M$, we will in addition assume
that we have given an almost complex structure $J$ on $M$. The concept
of c--projective compactness is then defined for linear connections on
$TM$ which preserve $J$ and are minimal in the sense that their
torsion is of type $(0,2)$ (so they are torsion--free if $J$ is
integrable). Such a connection $\nabla$ is then called c--projectively
compact if certain explicit c--projective modifications of $\nabla$,
constructed from local defining functions for the boundary, admit
smooth extensions to the boundary, see Definition \ref{def2.2}. 

If $\nabla$ is c--projectively compact, then the definition easily
implies that the almost c--projective structure defined by $\nabla$ admits a
smooth extension to the boundary. In particular, this provides a
smooth extension of $J$ to the boundary which endows $\partial M$ with
an almost CR structure, which is CR for integrable $J$. The main case
of interest will be that this structure is (Levi--)
non--degenerate. Moreover, we will always assume that, along the
boundary, the Nijenhuis tensor has values tangent to the boundary,
which in particular implies that the boundary structure is partially
integrable. This condition emerges naturally from several points of
view. Similar to the case of projective compactness, we show that for
a complex connection $\nabla$ which preserves a volume form, in addition to
requiring that the almost c--projective structure defined by $\nabla$ admits a
smooth extension to the boundary, one only has to require a certain
uniform rate of volume growth to ensure that $\nabla$ is
c--projectively compact, see Proposition \ref{prop2.3}.

In the integrable case, the Levi--Civita connection $\nabla$ of any
(pseudo--)K\"ahler metric $g$ on $M$ preserves $J$ and is torsion
free. In this setting we say $g$ is c--projectively compact if $\nabla$ is
c--projectively compact. In the non--integrable situation, a similar
concept is defined provided that $g$ is quasi--K\"ahler  which implies
that there is a linear connection, which preserves both $g$ and $J$
and is minimal, see Proposition \ref{prop2.5}. This so--called
\textit{canonical connection} is different from the Levi--Civita
connection of $g$ and plays the main role in questions related to
c--projective compactness. The main results of this article 
concern c--projectively compact metrics. In particular, we give two
equivalent characterizations of c--projective compactness of
quasi--K\"ahler metrics. 

For one of these equivalent characterizations, we assume that the
almost complex structure $J$ admits a smooth extension to all of
$\barm$, so as discussed above $\partial M$ inherits an almost CR
structures. We further assume that this structure is non--degenerate
and that the Nijenhuis tensor has asymptotically tangential
values. Under these assumptions, we devise a specific asymptotic form
for a metric $g$ involving the boundary geometry, see Section
\ref{2.6} for details. In Theorem \ref{thm2.6} we prove directly that
this asymptotic form is sufficient for c--projective compactness. Of
course this form may be used to provide local examples. In particular
this implies that the standard construction of complete K\"ahler
metrics, from boundary defining functions, on non--degenerate smoothly
bounded domains in $\mathbb{C}^{n}$ always leads to c--projectively
compact K\"ahler metrics; see Proposition \ref{prop2.7}. This clearly
demonstrates the richness and relevance of the class of c-projectively
compact metrics.

The second equivalent description of c--projective compactness for
metrics is based on a c--projective interpretation of scalar
curvature. Having given $\barm$, an almost complex structure $J$ and a
quasi--K\"ahler metric $g$ on $M$ with canonical connection $\nabla$,
assume that the almost c--projective structure defined by $\nabla$ admits a
smooth extension to $\barm$. (This condition can be easily checked in
local frames, see Lemma \ref{lemma4.2}.) After giving a c--projective
interpretation of several quantities associated to $g$, in particular
the scalar curvature $S$ of $\nabla$, we prove that these quantities
admit a smooth extension to $\barm$. The second equivalent condition
is then the extendability of the almost c--projective structure together with
the fact that resulting smooth extension of $S$ is nowhere vanishing
along $\partial M$.

For the proof of equivalence of these conditions to c--projective
compactness we use powerful tools from almost c--projective geometry;
these are developed in Section 3. We construct descriptions of the
real tractor bundles $\Cal H$ and $\Cal H^*$ of Hermitean metrics on
the c--projective strandard tractor bundle $\Cal T$ and on its dual
$\Cal T^*$. We describe the canonical tractor connections on these
bundles, the associated BGG splitting operators and the induced
invariant differential equations. Once one has obtained an extension
of the almost c--projective structure to the boundary, all these
objects admit extensions to the boundary, which is a major ingredient
in our proofs.

The proofs for the equivalence are then carried out in Section
\ref{4}. The basic extension results for quantities associated to a
quasi--K\"ahler metric on $(M,J)$, for which the almost c--projective
structure induced by the canonical connection extends to $\barm$, are
proved in Corollary \ref{cor4.1}. In Theorem \ref{thm4.3}, it is shown
that locally around boundary points for which the extension of $S$ is
non--vanishing, $g$ is then neccesarily c--projectively compact. The
remaining parts of the equivalence are: first a proof that a almost c--projectively
compact metric admits a weaker version of the asymptotic form
involving the restriction of $S$ to the boundary, see Propositon
\ref{prop4.4};  and then the proof that the boundary value of $S$ is locally
constant and finally that the asymptotic form has all required properties, see 
Theorem \ref{thm4.5}.

In the last part of the article, we analyze the curvature of
c--projectively compact connections and metrics. In particular, we
prove that the canonical connection $\nabla$ of a c--projectively
compact metric always satisfies an asymptotic version of the Einstein
equation and in the integrable case we give a complete description of
the curvature of the metric up to terms which admit a smooth extension
to the boundary, see Theorem \ref{thm4.7}. Finally, we show that
assuming an asymptotic vanishing condition for the covariant
derivative of the Nijenhuis tensor, one may, in many of our results,
replace curvarture quantities associated to the canonical connection
by the corresponding curvature quantities associated to the Levi--Civita connection, see
Corollary \ref{cor4.8}. 

It is to be expected that, similar to the projective case (see
\cite{proj-comp2}), one can use the c--projective standard tractor
bundle and its canonical tractor connection to obtain a description of
the tractors associated to the partially integrable almost CR
structure on the boundary. This will be taken up elsewhere.

\section{c--Projective compactness with a real boundary}\label{2}

\subsection{Almost c--projective structures}\label{2.1} 
Almost c--projective structures are the natural almost complex analog
of classical projective structures. In the literature, these
geometries are often referred to as $h$--projective or holomorphic
projective structures, but this is a misleading name. Indeed, there is
a holomorphic version of classical projective structures, i.e.~one
considers complex manifolds with holomorphic linear connections on
their tangent bundles which have the same complex geodesics up to
parametrization. This, however, is only a special case of
c--projective structures, which moreover is not relevant for many
important applications that involve Hermitean metrics. Therefore,
following \cite{CEMN}, we use the name ``(almost) c--projective'' for
these geometries.

Suppose that $M$ is a smooth manifold of even dimension $n=2m\geq 4$
endowed with an almost complex structure $J$ and suppose that $\nabla$
is a linear connection on $TM$ such that $J$ is parallel for the
connection induced by $\nabla$. This reads as $\nabla_\xi
J\eta=J\nabla_\xi\eta$ for all vector fields $\xi,\eta\in\frak
X(M)$. From this it is easy to see that changing $\nabla$ projectively
(in the real sense) will never lead to a connection which again
preserves $J$. The appropriate modification needed to preserve this
property is to take the usual definition of a projective change,
extend one--forms to complex linear functionals on the tangent space
and then use the complex analog of the classical
definition. Explicitly, this means that given $\nabla$ and a one--form
$\Up\in\Om^1(M)$, one defines a new linear connection
$\hat\nabla=\nabla+\Up$ by
\begin{equation}
\label{cpdef}
\hat\nabla_\xi\eta=\nabla_\xi\eta+\Up(\xi)\eta-\Up(J\xi)J\eta+
\Up(\eta)\xi-\Up(J\eta)J\xi.
\end{equation}
One immediately verifies that $\nabla J=0$ implies $\hat\nabla
J=0$. One defines two linear connections on $TM$ which preserve $J$ to
be \textit{c--projectively equivalent} if and only if they are related
by \eqref{cpdef} for some one--form $\Up\in\Om^1(M)$.

Let us remark at this point, that while \eqref{cpdef} is the
traditional way to associate a c--projective change of connection to a
one--form, it does not agree with the conventions used in \cite{CEMN},
see Section 2.1 of that reference. The conventions in \cite{CEMN} are
designed to work well in a complexified setting. For the current
paper, however, it is of crucial importance to work in a real
setting. In this real setting, the traditional conventions lead to
formulae which are more closely analogous to the case of projective
geometry and thus simplify comparison to the theory of projective
compactness developed in \cite{proj-comp,proj-comp2,scalar}. 

From the definition in \eqref{cpdef} it is evident that
$\hat\nabla_\xi\eta-\nabla_\xi\eta$ is symmetric in $\xi$ and $\eta$,
which implies $\nabla$ and $\hat\nabla$ have the same torsion. Now it
is well known that for a connection preserving an almost complex
structure the $(0,2)$--component of the torsion (i.e.~the part which
is conjugate linear in both arguments) is independent of the
connection and equals (up to a non--zero factor) the Nijenhuis--tensor
of the almost complex structure. On the other hand, the other
components of the torsion can be removed by a change of the connection
(preserving $J$); hence there always are linear connections which
preserve $J$ and whose torsion is of type $(0,2)$. This motivates the
following definitions.

\begin{definition}\label{def2.1}
  Let $(M,J)$ be an almost complex manifold.

  (1) A linear connection $\nabla$ on $M$ is called \textit{complex}
  if $J$ is parallel for (the connection induced by) $\nabla$ and it
  is called \textit{minimal} if its torsion is of type $(0,2)$,
  i.e.~conjugate linear in both arguments.

  (2) An \textit{almost c--projective structure} on $M$ is a
  c--projective equivalence class $[\nabla]$ of minimal complex linear
  connections on $TM$.

  (3) The structure is called \textit{c--projective} or
  \textit{torsion free} if and only if $J$ is integrable or
  equivalently the connections in the projective class are torsion
  free.
\end{definition}

As in the case of usual projective structures, a linear connection on
$TM$ induces linear connections on all natural vector bundles over
$M$, i.e.~on all vector bundles induced from the complex linear frame
bundle of $M$. In particular, one can form real and complex density
bundles, which we will frequently need in what follows. Let us fix the
conventions we will use. We will denote the complex line bundle
$\La^m_{\Bbb C}TM$ (the highest complex exterior power of the tangent
bundle) by $\Cal E(m+1,0)$. We will assume that there exist $(m+1)$st
roots of this line bundle, and that a specific root $\Cal E(1,0)$ has
been chosen. We then define $\Cal E(-1,0)$, $\Cal E(0,1)$, and $\Cal
E(0,-1)$ as the dual, the conjugate, and the conjugate dual bundle to
$\Cal E(1,0)$. Forming tensor powers, we thus obtain complex density
bundles $\Cal E(k,\ell)$ for $k,\ell\in\Bbb Z$.

On the other hand, we can consider the usual real density bundles. As
an almost complex manifold, $M$ is automatically orientable and the
almost complex structure induces an orientation. Thus the bundle of
real volume densities naturally includes into $\Cal E(-m-1,-m-1)$, so
we denote it by $\Cal E(-2m-2)$. Since this is a trivial real line
bundle, we can form arbitrary real roots of this bundle, thus defining
$\Cal E(w)$ for all $w\in\Bbb R$. Passing to roots, we obtain
inclusions $\Cal E(2k)\subset\Cal E(k,k)$ for all $k\in\Bbb Z$. This
then allows us to define complex density bundles $E(w,w')$ for all
$w,w'\in\Bbb R$ provided that $w-w'\in\Bbb Z$, and we always get an
inclusion of the real line bundle $\Cal E(w)$ into the complex line
bundle $\Cal E(\frac{w}2,\frac{w}2)$. We will follow the convention
that adding $(w)$ to the name of a real vector bundle indicates a
tensor product with $\Cal E(w)$, while for a complex vector bundle,
adding $(w,w')$ to the name indicates a complex tensor product with
$\Cal E(w,w')$.

Almost c--projective structures with the additional choice of a line
bundle $\Cal E(1,0)$ can be equivalently described as (real) normal
parabolic geometries of type $(G,P)$, where $G:=SL(m+1,\Bbb C)$ and
$P$ is the stabilizer of a complex line in the standard representation
$\Bbb C^{m+1}$ of $G$, see \cite{CEMN}. This correspondence is
established similarly to the discussion of the real case in Section
4.1.5 of \cite{book}, see also Section 4.6 of \cite{twistors}.

This Cartan geometry can be equivalently encoded via the associated
bundle corresponding to the standard representation $\Bbb C^{m+1}$ of
$G$. This is the so--called \textit{standard tractor bundle} $\Cal T$
on which the Cartan connection induces a linear connection
$\nabla^{\Cal T}$, the \textit{standard tractor connection}. This
approach is developed in detail (mainly in a complexified picture) in
\cite{CEMN}.

\subsection{The notion of c--projective compactness}\label{2.2}
We now introduce a notion analogous to projective compactness of order
two as introduced in \cite{ageom} and \cite{proj-comp}, and further
studied in \cite{proj-comp2} and \cite{scalar}. Let $\barm$ be a real
smooth manifold of real dimension $n=2m$ with boundary $\partial M$
and interior $M$. Suppose further that we have given an almost complex
structure $J$ on $M$ and a minimal complex linear connection $\nabla$
on $TM$.

\begin{definition}\label{def2.2}
  The complex connection $\nabla$ on $TM$ is called \textit{c--projectively
    compact} if and only if for each $x\in\partial M$, there is a
  neighborhood $U$ of $x\in\barm$ and a smooth defining function
  $\rho:U\to\Bbb R_{\geq 0}$ for $U\cap\partial M$ such that the
  c--projectively equivalent connection $\hat\nabla
  =\nabla+\frac{d\rho}{2\rho}$ on $U\cap M$ extends smoothly to all of
  $U$.
\end{definition}

Observe first that this condition is independent of the
defining function under consideration. Given a defining function
$\rho$ on $U$, any other defining function on $U$ can be written as
$\hat\rho=e^f\rho$ for some smooth function $f:U\to\Bbb R$. This
immediately implies that $d\hat\rho=\hat\rho df+e^fd\rho$ and hence
$\frac{d\hat\rho}{\hat\rho}=df+\frac{d\rho}{\rho}$. Since $df$ is smooth
up to the boundary, this implies that if the connection associated to
$\rho$ extends, then so does the one associated to $\hat\rho$.

Using this, we can prove a first nice result on c--projectively
compact connections.
\begin{prop}\label{prop2.2}
  If the linear connection $\nabla$ on $TM$ is c--projectively
  compact, then the almost c--projective structure $(J,[\nabla])$ on
  $M$ naturally extends to all of $\barm$. In particular, the almost
  complex structure $J$ smoothly extends to $\barm$, which gives rise
  to a (possibly degenerate) almost CR structure of hypersurface type
  on $\partial M$. This structure is integrable (CR) if the initial
  structure is c--projective.
\end{prop}
\begin{proof}
  As we have noted in \ref{2.1}, the linear connection
  $\hat\nabla=\nabla+\frac{d\rho}{2\rho}$ on $U\cap M$ satisfies
  $\hat\nabla J=0$. But since this connection extends smoothly to all
  of $U$, we can extend $J$ by parallel transport to all of $U$. Since
  $U\cap M$ is dense in $U$, this extension is uniquely determined by
  $J$ and satisfies $J\o J=-\id$ on all of $U$. If the initial almost
  complex structure $J$ is integrable then the same is true for the
  extended structure (since its Nijenhuis tensor vanishes on a dense
  subset).

  For $x\in\partial M$, we define $H_x:=T_x\partial M\cap
  J(T_x\partial M)$. Clearly, this defines a smooth distribution
  $H\subset T\partial M$ of corank one and the boundary value of $J$
  defines an almost complex structure on this distribution. Hence we
  have obtained an almost CR structure of hypersurface type on
  $\partial M$. It is well known that this structure is CR if $J$ is
  integrable.

  As we have seen above, any other defining function $\hat\rho$ can be
  written as $\hat\rho=e^f\rho$ for a function $f$ which is smooth up
  to the boundary and then
  $\tfrac{d\hat\rho}{2\hat\rho}=\tfrac{d\rho}{2\rho}+\tfrac12
  df$. Hence the connections associated to $\hat\rho$ and $\rho$ are
  c--projectively equivalent on all of $U$. Together with the
  extension of $J$, we thus get a well defined almost c--projective
  structure on $\barm$.
\end{proof}

Observe that to obtain the induced almost CR structure on the
boundary, one only needs the almost complex structure to extend, the
c--projective equivalence class of connections is not really used at
this point.

\subsection{On the boundary geometry}\label{2.2a} 
In the case of an integrable complex structure, we always obtain a CR
structure on the boundary $\partial M$. From the point of view of CR
geometry, it is natural to assume in addition that the structure is
(Levi--) non--degenerate. If the initial almost complex structure $J$
is not integrable, then, apart from non--degeneracy, some assumptions
on the asymptotics of the almost complex structure must be made. A
natural assumption would be partial integrability of the induced CR
geometry, but it turns out that for our purposes a slightly stronger
assumption will be suitable.

To formulate the necessary definitions, recall first that the Lie
bracket of vector fields induces a bilinear operator
$\Ga(H)\x\Ga(H)\to \Ga(T\partial M/H)$. This operator is immediately
seen to be bilinear over smooth functions and thus it is induced by a bundle
map $\Cal L:H\x H\to T\partial M/H$, which is called the
\textit{Levi--bracket}. 

\begin{definition}\label{def2.2a}
Let $\barm$ be a smooth manifold with boundary $\partial M$ and
interior $M$ and let $J$ be an almost complex structure on
$\barm$. Let $H\subset T\partial M$ be the induced almost CR structure
and $\Cal L$ its Levi--bracket. 

(1) The almost CR structure $H\subset T\partial M$ is called
  \textit{non--degenerate} if the value of $\Cal L$ at any point is a
  non--degenerate bilinear map.

(2) The structure is called \textit{partially integrable} if $\Cal L$
  is Hermitean in the sense that $\Cal L(\xi,\eta)=\Cal L(J\xi,J\eta)$
  for all $\xi,\eta\in H$.

(3) We say that the Nijenhuis tensor $\Cal N$ of $J$ has
  \textit{asymptotically tangential values} if, along $\partial M$,
  $\Cal N$ has values in $T\partial M\subset T\barm|_{\partial M}$.
\end{definition}

Now we can easily characterize these conditions in terms of local defining
functions.

\begin{lemma}\label{lem2.2a}
  Let $(\barm,J)$ be an almost complex manifold with boundary
  $\partial M$ and interior $M$, and let $\Cal N$ be the
  Nijenhuis tensor of $J$. For a local defining function $\rho$
  for the boundary, put $\th=-d\rho\o J$. 

  (1) The fact that $\Cal N$ has asymptotically tangential values is
  equivalent to either of the following two conditions on any local
  defining function $\rho$. 
\begin{itemize}
\item The exterior derivative $d\th$ is Hermitean on
  $T\barm|_{\partial M}$. 
\item For any minimal complex connection $\nabla$ on $\barm$, the
  $\binom02$--tensor field $\nabla d\rho$ is symmetric on
  $T\barm|_{\partial M}$. 
\end{itemize}

  (2) The induced almost CR structure on $\partial M$ is
non--degenerate if and only if for any local defining function $\rho$
the one--form $\th$ restricts to a contact form on $\partial M$. The
structure is partially integrable if and only if the restriction of
$d\th$ to $H\subset T\partial M$ is Hermitean. 
\end{lemma}
\begin{proof}
By construction, for each $x\in\partial M$, the kernel of
$\th(x)|_{T_x\partial M}:T_x\partial M\to\Bbb R$ coincides with
$H_x$. For sections $\xi,\eta\in\Ga(H)$ we thus get
$d\th(\xi,\eta)=-\th([\xi,\eta])$, so the restriction of $d\th$ to
$H\x H$ represents the Levi--bracket $\Cal L$. Since non--degeneracy
of the restriction of $d\th$ to $\ker(\th)$ is equivalent to $\th$
being a contact form, this implies (2).

(1) Since the Nijenhuis tensor is conjugate linear in both arguments,
its values at each point $x\in\barm$ form a complex subspace of
$T_x\barm$. Hence the fact that $\Cal N$ has asymptotically tangential
values is equivalent to $\Cal N$ having values in $H$ along the
boundary. This is in turn equivalent to the insertion of $\Cal N$ into
$d\rho$ or into $\th$ vanishing along $\partial M$.

Now if $\nabla$ is a complex connection on $T\barm$, then we can
compute $0=dd\rho$ as the sum of the alternation of $\nabla d\rho$ and
a term in which the torsion of $\nabla$ is inserted into $d\rho$. If
$\nabla$ is minimal, this torsion is a non--zero multiple of $\Cal N$,
which implies the equivalence to the second condition.

To complete the proof, we show that for all $\xi,\eta\in\frak
X(\barm)$ we get
$$
d\th(J\xi,\eta)+d\th(\xi,J\eta)=d\rho(\Cal N(\xi,\eta)).
$$ 
Since both sides are bilinear over smooth functions, we may without
loss of generality assume that $d\rho(\xi)=\th(J\xi)$ 
and $\th(\xi)=-d\rho(J\xi)$ are constant and likewise for
$\eta$. Assuming this, we get
$d\th(J\xi,\eta)=-\th([J\xi,\eta])=d\rho(J[J\xi,\eta])$ and similarly
for the second term in the left hand side. On the other hand, we also
see that $d\rho([\xi,\eta])=-dd\rho(\xi,\eta)=0$ and likewise 
$d\rho([J\xi,J\eta])=0$, so the claim follows from the definition of
the Nijenhuis tensor.
\end{proof}

\subsection{Volume asymptotics}\label{2.3} 
We next analyze the effect of a c--projective change of connection on
the induced connections on density bundles. From the interpretation of
the c--projective change law via complex linear extensions of
one--forms, one easily concludes that on the top exterior power
$\La^m_{\Bbb C}TM$, the change of connection is given by
$$
\hat\nabla_\xi s=\nabla_{\xi}s+(m+1)(\Up(\xi)-i\Up(J\xi))s,
$$
which also explains the conventions for density bundles we have
chosen. This immediately implies that for $w,w'\in\Bbb R$ with
$w-w'\in\Bbb Z$ and $s\in\Ga(\Cal E(w,w'))$ we have
\begin{equation}
  \label{eq:cdens-trans}
  \hat\nabla_\xi s=\nabla_{\xi}s+((w+w')\Up(\xi)-(w-w')i\Up(J\xi))s,
\end{equation}
while for $\si\in\Ga(\Cal E(w))$ with $w\in\Bbb R$ we obtain
\begin{equation}
  \label{eq:rdens-trans}
\hat\nabla_\xi\si=\nabla_{\xi}\si+w\Up(\xi)\si.   
\end{equation}

Having these results at hand, the relation between c--projective
compactness and volume asymptotics can be analyzed as it is done for
projective compactness in \cite{proj-comp}. We call a connection
$\nabla$ \textit{special} if and only if there is a non--vanishing
section $\si$ of some (or equivalently) any real density bundle $\Cal
E(w)$ with $w\neq 0$, which is parallel for $\nabla$. The notion of
volume asymptotics as introduced in Definition 2.2 of \cite{proj-comp}
can be used in our setting. Also, the relation to (real) defining
densities continues to hold:

\begin{prop}\label{prop2.3}
  let $\barm$ be a smooth manifold of real dimension $2m$ with
  boundary $\partial M$ and interior $M$, and let $\nabla$ be a
  special linear complex connection on $TM$.

  (1) If $\nabla$ is c--projectively compact, then it has volume
  asymptotics of order $m+1$ in the sense of Section 2.2 of
  \cite{proj-comp}. Moreover, any non--zero section of $\Cal E(2)$
  which is parallel for $\nabla$ extends by zero to a defining density
  for $\partial M$.

  (2) Conversely, assume that the almost c--projective structure on $M$
  defined by $\nabla$ admits a smooth extension to $\barm$ and that
  there is a defining density $\tau\in\Ga(\Cal E(2))$ for $\partial M$
  such that $\tau|_M$ is parallel for $\nabla$. Then $\nabla$ is
  c--projectively compact.
\end{prop}
\begin{proof}
With some trivial modifications, the proof of Proposition 2.3 of
\cite{proj-comp} applies. 
\end{proof}

\subsection{C--projective compactness for metrics}\label{2.5}
In the setting of projective compactness, one always deals with
torsion--free connections, so projective compactness of a
pseudo--Riemannian metric can be defined as projective compactness of
its Levi--Civita connection. Looking for a c--projective analog of
this concept, we will only consider pseudo--Riemannian metrics $g$
which are Hermitean with respect to the given almost complex
structure, i.e.~such that $g(J\xi,J\eta)=g(\xi,\eta)$ for all $\xi$
and $\eta$. In order to get a sensible concept, we have to associate
to $g$ a connection which is complex and minimal. The latter two
conditions specify the torsion of the connection, so if the connection
also is required to preserve $g$, it is uniquely determined by these
properties (if it exists). This motivates the following definitions.

\begin{definition*}
(1) For an almost complex manifold $(N,J)$, a pseudo--Riemannian
  metric $g$ which is Hermitian for $J$ is called \textit{admissible}
  if and only if there is a linear connection $\nabla$ on $TN$ which
  is minimal and preserves both $J$ and $g$. If such a connection
  exists then we know from above that it is uniquely determined and we
  call it the \textit{canonical connection} associated to $g$. 

  (2) Consider a smooth manifold $\barm$ with boundary $\partial M$
  and interior $M$, and an almost complex structure $J$ on $M$. An
  admissible Hermitean metric $g$ on $(M,J)$ is called
  \textit{c--projectively compact} if and only if its canonical
  connection is c--projectively compact in the sense of Definition
  \ref{def2.2}.
\end{definition*}

Now we can give a description of admissible metrics in terms of the
Gray--Hervella classification (see \cite{G-H}) of almost Hermitean
structures. 

\begin{prop}\label{prop2.5}
  Let $g$ be a Hermitean pseudo--Riemannian metric on an almost
  complex manifold $(N,J)$. Then $g$ is admissible if and only if it
  is quasi--K\"ahler (i.e.~of class $\Cal W_1\oplus\Cal W_2$) in the
  sense of \cite{G-H}. 

In particular, nearly K\"ahler metrics (of arbitrary signature) are
admissible, and if $J$ is integrable, then $g$ is admissible if and
only if it is (pseudo--)K\"ahler.
\end{prop}
\begin{proof}
Denoting by $\om$ the fundamental two--form of $g$, so
$\om(\xi,\eta):=-g(\xi,J\eta)$, the Hervella--Gray classification is
based on $\nabla^g\om$, where $\nabla^g$ denotes the Levi--Civita
connection of $g$. Now it is well known that there is a connection
$\nabla$ on $N$ for which $\nabla g=0$ and $\nabla J=0$ and thus
$\nabla\om=0$. For such a connection, consider the contorsion,
i.e.~the $\binom12$--tensor field $A$ defined by
$\nabla_\xi\eta=\nabla^g_\xi \eta+A(\xi,\eta)$. Since both $\nabla^g$
and $\nabla$ preserve $g$, we conclude that
$g(A(\xi,\eta),\ze)=-g(A(\xi,\ze),\eta)$ for all $\xi,\eta,\ze\in\frak
X(M)$.

Moreover, if we add a $\binom12$--tensor field to $A$ which which has
the same skew symmetry property and in addition is complex linear in
the second variable, this will also lead to a connection which
preserves both $g$ and $J$. This means that we can subtract the
complex linear part
$A_+(\xi,\eta):=\tfrac12(A(\xi,\eta)-JA(\xi,J\eta))$ from $A$, without
loosing the property that the resulting connection $\nabla$ preserves
both $g$ and $J$. Hence from now on we assume without loss of
generality that $A(\xi,J\eta)=-JA(\xi,\eta)$. 

Now we compute $\nabla^g\om(\xi,\eta,\ze)$ as 
$$
-\xi\cdot
g(\eta,J\ze)+g(\nabla^g_\xi\eta,J\ze)+g(\eta,J\nabla^g_\xi\ze).
$$ 
Rewriting the last summand as $-g(J\eta,\nabla^g_\xi\ze)$ we can
express each $\nabla^g$ as $\nabla+A$ and using that $\nabla J=0$ and
$\nabla g=0$, we obtain 
$$
\nabla^g\om(\xi,\eta,\ze)=g(A(\xi,\eta),J\zeta)-g(J\eta,A(\xi,\ze)).
$$ 
The last term (including the sign) can be written as
$g(A(\xi,J\eta),\ze)$ and using conjugate linearity of $A$ in the
second variable, we end up with 
$$
\nabla^g\om(\xi,\eta,\ze)=2g(A(\xi,\eta),J\zeta). 
$$
According to Theorem 3.1 in \cite{G-H}, $g$ is quasi--K\"ahler if and
only if 
$$
\nabla^g\om(\xi,\eta,\ze)+\nabla^g\om(J\xi,J\eta,\ze)=0
$$ 
for all $\xi,\eta,\ze$. This is clearly equivalent to
$A(\xi,\eta)+A(J\xi,J\eta)=0$ for all $\xi,\eta\in\frak X(N)$. But this
implies that $A(J\xi,\eta)=A(\xi,J\eta)=-JA(\xi,\eta)$, so $A$ must be
conjugate linear in both variables. Hence
$T(\xi,\eta)=A(\xi,\eta)-A(\eta,\xi)$ is also conjugate linear in both
variables, and of course, $T$ is the torsion of $\nabla$. Hence $T$
must coincide with $-\tfrac14\Cal N$. This shows that any
quasi--K\"ahler metric is admissible and that in the integrable case,
$\nabla=\nabla^g$, so $g$ is K\"ahler.

Conversely, if $g$ is admissible and $\nabla$ is its canonical
connection with torsion $T=-\frac14\Cal N$, then one can explicitly
compute the contorsion tensor $A$ via 
$$
g(A(\xi,\eta),J\zeta)=\tfrac12(-g(T(\xi,J\ze),\eta)+
g(T(\xi,\eta),J\ze)-g(T(\eta,J\ze),\xi)). 
$$
Using that $T$ is conjugate linear in both arguments, one immediately
verifies that the right hand side changes sign if one either replaces
$(\xi,\eta)$ by $(J\xi,J\eta)$ or $(\eta,\ze)$ by $(J\eta,J\ze)$. Of
course, this implies that $A(\xi,\eta)+A(J\xi,J\eta)=0$ and that $A$
is conjugate linear in the second variable, so $g$ is quasi--K\"ahler.
\end{proof}

\subsection{A sufficient condition for c--projective
  compactness}\label{2.6} Motivated by the results for projective
compactness in \cite{proj-comp} and \cite{proj-comp2}, we describe an
asymptotic form for a Hermitean metric which is sufficient for
c--projective compactness. We start with a manifold
$\barm=M\cup\partial M$ with boundary, which is endowed with an almost
complex structure $J$. As observed in \ref{2.2} the induces an almost
CR structure on $\partial M$, which is assumed to be
non--degenerate. Moreover, we assume that the Nijenhuis tensor $\Cal
N$ has asymptotically tangential values, which by Lemma \ref{lem2.2a}
implies that the almost CR structure on $\partial M$ is partially
integrable.

Now suppose that $\rho$ is a local defining function for $\partial M$,
and consider the one--form $\th:=-d\rho\o J$ as in \ref{2.2a}, which
in our current setting is smooth up to the boundary. By Lemma
\ref{lem2.2a} our assumptions imply that $\th$ restricts to a contact
form on $\partial M$ and that $d\th$ is a Hermitean on
$T\barm|_{\partial M}$ and non--degenerate on $H$. In particular, the
restriction of $d\th$ to $H$ is the imaginary part of a
non--degenerate Hermitean form, so it has a well defined signature
$(p,q)$. Now we assume that $g$ is a pseudo--Riemannian metric on the
interior $M$ which is Hermitean for $J$ and has signature $(p+1,q)$ or
$(p,q+1)$. The asymptotic form we consider is that, locally near the
boundary, we can write
\begin{equation}
  \label{asymp-form}
  g=C\left(\frac{d\rho^2}{\rho^2}+\frac{\th^2}{\rho^2}\right)+\frac{h}{\rho}.  
\end{equation}
Here $h$ is a Hermitean form which admits a smooth extension to the
boundary such that, along the boundary, we have
$h(\xi,J\ze)=Cd\th(\xi,\ze)$, whenever $\ze$ lies in the CR
subspace. In particular, the boundary value of $h$ is non--degenerate
on the CR subspace.

Similarly to the case of projective compactness of order $2$, this
asymptotic form is independent of the defining function. From
\ref{2.2}, we see that for $\hat\rho=e^f\rho$ we get
$\tfrac{d\hat\rho}{\hat\rho}=\tfrac{d\rho}{\rho}+df$ and
$\tfrac{\hat\th}{\hat\rho}=\tfrac{\th}{\rho}-df\o J$. Using this, a
simple direct computation shows that an asymptotic form as in
\eqref{asymp-form} with respect to $\rho$ implies an analogous form
with respect to $\hat\rho$ with the same constant $C$ and with 
$$
\hat h=e^fh+2C\left(-df\odot d\hat\rho+(df\o
  J)\odot\hat\th\right)+C\hat\rho\left(df^2+(-df\o J)^2\right).
$$
On the other hand, applying the exterior derivative to the formula
for $\hat\th$, one gets
$$
d\hat\th=e^fd\th+df\wedge\hat\th-d\hat\rho\wedge(df\o J)-\hat\rho
d(df\o J).
$$
Along the boundary and for $\ze$ in the CR subspace, we thus get 
$$
\hat h(\xi,J\zeta)=e^fh(\xi,J\zeta)-Cdf(J\zeta)d\hat\rho(\xi)-C 
df(\zeta)\hat\th(\xi),
$$
and this coincides with $Cd\hat\th(\xi,\zeta)$.

Now we can prove our first main result, namely that such a form is
sufficient for c--projective compactness.

\begin{thm}\label{thm2.6}
  Let $\barm$ be a smooth manifold with boundary $\partial M$ and
  interior $M$. Let $J$ be an almost complex structure on $\barm$,
  such that $\partial M$ is non--degenerate and the Nijenhuis tensor
  $\Cal N$ of $J$ is asymptotically has tangential values. Let $g$ be
  an admissible pseudo--Riemannian Hermitean metric on $M$.

  For a local defining function $\rho$ for the boundary defined on an
  open subset $U\subset\barm$, put $\th=-d\rho\o J$ and, given a
  constant $C$, define a Hermitean $\binom02$--tensor field
  $h_{\rho,C}$ on $U\cap M$ by
$$
h_{\rho,C}(\xi,\eta):=\rho g(\xi,\eta)
-\tfrac{C}{\rho}\left(d\rho(\xi)d\rho(\eta)+\th(\xi)\th(\eta)\right). 
$$
  Suppose that for each $x\in\partial M$ there are an open neighborhood
  $U$ of $x$ in $\barm$, a local defining function $\rho$ defined on
  $U$, and a non--zero constant $C$ such that
\begin{itemize}
\item $h_{\rho,C}$ admits a smooth extension to all of $U$
\item for all $\xi,\ze\in\frak X(U)$ with $d\rho(\ze)=\th(\ze)=0$, the function
$h(\xi,J\ze)$ approaches $Cd\th(\xi,\ze)$ at the boundary.
\end{itemize}
   Then $g$ is c--projectively compact.
\end{thm}
\begin{proof}
  This is parallel to the proof for projective compactness in Theorem
  2.6 of \cite{proj-comp}, and we will partly refer to that proof and
  emphasize the differences. 

  Let $\nabla$ be the canonical connection for $g$ and let
  $\hat\nabla$ be the c--projectively modified connection
  corresponding to $\Up=\frac{d\rho}{2\rho}$. We have to show that for
  arbitrary vector fields $\xi$ and $\eta$ which are smooth on all of
  $U$, also $\hat\nabla_\xi\eta$ admits a smooth extension from $U\cap
  M$ to all of $U$. In order to prove this, it suffices to do the
  following. We first show $d\rho(\hat\nabla_\xi\eta)$ admits a smooth
  extension to the boundary. Next, we prove that for any $\ze\in\frak
  X(U)$ such that $d\rho(\ze)$ and $\th(\ze)$ vanish identically,
  $h(\hat\nabla_\xi\eta,\ze)$ admits a smooth extension to all of
  $U$. Since we can write $\th(\hat\nabla_\xi\eta)$ as
  $-d\rho(\hat\nabla_\xi J\eta)$, these two facts imply that the
  coordinate functions of $\hat\nabla_\xi\eta$ with respect to an
  appropriate frame for $T\barm|_U$ admit a smooth extension to the
  boundary and hence the result.

  Shrinking $U$ if necessary, we may assume that, on $U$, $d\rho$ is
  nowhere vanishing and $h$ is non--degenerate on
  $\ker(d\rho)\cap\ker(\th)$.  Extend $d\rho$ and $\th$
  to a coframe for $\barm$ on $U$ and take the first element
  $\tilde\ze_0$ in the dual frame. Then non--degeneracy of $h$ on
  $\ker(d\rho)\cap\ker(\th)$ implies that we can add a section of this
  subbundle to $\tilde\ze_0$ to obtain a vector field $\ze_0\in\frak
  X(U)$ such that $d\rho(\ze_0)\equiv 1$, $\th(\ze_0)\equiv 0$ and
  $h(\xi,\ze_0)=0$, whenever $\xi\in\frak X(U)$ satisfies
  $d\rho(\xi)=\th(\xi)=0$.

  Now an arbitrary vector field $\xi\in\frak X(U)$ can be written as
  $\xi=d\rho(\xi)\ze_0+\th(\xi)J\ze_0+\xi_H$, where
  $d\rho(\xi_H)=\th(\xi_H)=0$. Inserting into the defining equation
  for $h$ and multiplying by $\rho$, we conclude that
  $\rho^2g(\xi,\ze_0)=d\rho(\xi)(C+\rho h(\ze_0,\ze_0))$. In
  particular, we see that we can prove that
  $d\rho(\hat\nabla_\xi\eta)$ admits a smooth extension to all of $U$
  by showing that $\rho^2g(\hat\nabla_\xi\eta,\ze_0)$ admits such an
  extension.

  On the other hand, assume that $\ze\in\frak X(U)$ has the property
  that $d\rho(\ze)$ and $\th(\ze)$ vanish identically. Then inserting
  into the defining equation for $h$ we see that for each $\xi\in\frak
  X(U)$, we get $\rho g(\xi,\ze)=h(\xi,\ze)$. Hence we can show that
  for each such $\ze$, $\rho g(\hat\nabla_\xi\eta,\ze)$ admits a
  smooth extension to all of $U$ in order to verify the second claimed
  smoothness property. 

  As in the projective case, we use a modification of the
  Koszul--formula for the Levi--Civita connection in order to prove
  existence of smooth extensions. Since we are dealing with a metric
  connection with torsion here, the Koszul formula becomes a bit more
  complicated. Recall that to prove the Koszul formula, one uses
  torsion freeness of the Levi--Civita connection, which brings in the
  terms involving Lie brackets. In the presence of torsion, one simply
  has to add, for each Lie bracket term, a term involving the torsion
  with the same configuration of arguments. Taking into account that
  the torsion of a minimal complex connection equals $-\frac14\Cal N$,
  we conclude that for the canonical connection $\nabla$, we get  
\begin{align*}
  2g(\nabla_\xi\eta,\zeta)&=\xi\cdot g(\eta,\zeta)-\zeta\cdot
  g(\xi,\eta)+\eta\cdot
  g(\xi,\zeta)\\
  &+g([\xi,\eta],\zeta)-g([\xi,\zeta],\eta)-g([\eta,\zeta],\xi)\\
  &-\tfrac14g(\Cal N(\xi,\eta),\zeta)+\tfrac14g(\Cal
  N(\xi,\zeta),\eta)+\tfrac14g(\Cal
  N(\eta,\zeta),\xi).
\end{align*}
To compute  $2g(\hat\nabla_\xi\eta,\zeta)$ we have to add to this the
expression 
\begin{equation}\label{addon}
  \tfrac{d\rho(\xi)}{\rho}g(\eta,\zeta)+
  \tfrac{d\rho(\eta)}{\rho}g(\xi,\zeta)+
  \tfrac{\th(\xi)}{\rho}g(J\eta,\zeta)+
  \tfrac{\th(\eta)}{\rho}g(J\xi,\zeta)
\end{equation}
In the first step, we use this formula to compute
$2g(\hat\nabla_\xi\eta,\ze_0)$ for $\xi,\eta\in\frak X(U)$, and we
ignore terms which admit a smooth extension to the boundary after
multiplication by $\rho^2$. In particular, this applies to all terms
in the second and third line of the Koszul formula. Now from above we
know that $g(\eta,\ze_0)=\tfrac{d\rho(\eta)}{\rho^2}(C+\rho
h(\ze_0,\ze_0))$, and, even if we differentiate once more or multiply
by $\tfrac1\rho$, the second summand will admit a smooth extension
after multiplication by $\rho^2$. Hence in the computations, we may
replace $g(\eta,\ze_0)$ by $C\tfrac{d\rho(\eta)}{\rho^2}$ and
$g(J\eta,\ze_0)$ by $-C\tfrac{\th(\eta)}{\rho^2}$ and likewise for
$\xi$. In particular, the contributions from \eqref{addon} simply add
up to
$$
2C\tfrac1{\rho^3}\big(d\rho(\xi)d\rho(\eta)-\th(\xi)\th(\eta)\big). 
$$
On the other hand, in the terms from the first line of the Koszul
formula, we only have to take into account those parts in which a
vector field differentiates the factor $\tfrac{1}{\rho^2}$. This
immediately shows that the first and third terms each contribute
$-2C\tfrac1{\rho^3}d\rho(\xi)d\rho(\eta)$, while from the second term
we get
$$
+2C\tfrac1{\rho^3}(d\rho(\xi)d\rho(\eta)+\th(\xi)\th(\eta)).
$$ 
This completes the proof that $\rho^2g(\hat\nabla_\xi\eta,\ze_0)$
admits a smooth extension to the boundary.

For the second part, we take $\ze\in\frak X(U)$ such that $d\rho(\ze)$
and $\th(\ze)$ vanish identically and we can ignore terms which admit
a smooth extension to all of $U$ after multiplication by $\rho$. Then
$g(\eta,\zeta)=\tfrac1\rho h(\eta,\ze)$ and likewise for $\xi$, so the
contribution of the terms from \eqref{addon} reads as
\begin{equation}
  \label{add-contrib}
  \tfrac{1}{\rho^2}\big(d\rho(\xi)h(\eta,\zeta)+d\rho(\eta)h(\xi,\zeta)+
\th(\xi)h(J\eta,\zeta)+\th(\eta)h(J\xi,\zeta)\big).
\end{equation}
The first two summands in \eqref{add-contrib} are immediately seen to
cancel with the contributions coming from the first and third summand
in the right hand side of the Koszul formula. The term
$g([\xi,\eta],\ze)$ clearly admits a smooth extension to the boundary
after multiplication by $\rho$ and the same holds for all terms in the
Koszul formula which involve $\Cal N$ by the assumption on
asymptotically tangential values. Hence we are left with determining the
contribution of
$$
-\ze\cdot g(\xi,\eta)-g([\xi,\ze],\eta)-g([\eta,\ze],\xi). 
$$
By assumption $\ze\cdot\rho=0$, so in all terms we only need the
$\tfrac1{\rho^2}$--terms in the formula for $g$. Now $d\rho(\ze)=0$
implies $-\ze\cdot d\rho(\xi)-d\rho([\xi,\ze])=dd\rho(\xi,\zeta)=0$, and
since $\th(\ze)=0$, we similarly get
$-\ze\cdot\th(\xi)-\th([\xi,\ze])=d\th(\xi,\ze)$. This shows that the
total contribution from our three remaining terms is given by 
$$
C\tfrac1{\rho^2}\big(d\th(\xi,\zeta)\th(\eta)+d\th(\eta,\zeta)\th(\xi)\big).
$$ 
By our assumption on the relation between $h$ and $d\th$, this cancels
with the remaining two summands from \eqref{add-contrib}, which
completes the proof. 
\end{proof}

\subsection{A class of examples}\label{2.7}
Using Theorem \ref{thm2.6} we can now construct a class of examples of
c--projectively compact metrics which includes the classical examples
of complete K\"ahler metrics associated to smoothly bounded domains in
$\Bbb C^n$. 

Consider an almost complex manifold $(\Cal M,J)$ and a domain
$U\subset\Cal M$ with smooth boundary $\partial U$. Then $\partial U$
inherits an almost CR structure (see \ref{2.2}) and we assume that
this structure is Levi--non--degenerate. Choose a defining function
$\rho$ for $\partial U$ defined on a neighborhood of $\bar U$ in $\Cal
M$. Assuming that $\rho>0$ on $U$ the function $\log\rho:U\to\Bbb R$
is smooth. Now putting $\om=d(-d\log\rho\o J)\in\Om^2(U)$, we define
$g_\rho\in\Ga(S^2T^*U)$ by $g_{\rho}(\xi,\eta):=\om(\xi,J\eta)$ for
$\xi,\eta\in\frak X(U)$.

\begin{prop}\label{prop2.7}
  For a domain $U$ in an almost complex manifold $(\Cal M,J)$ with
  smooth, Levi non--degenerate boundary $\partial U$ consider the
  section $g_\rho$ associated to a defining function $\rho$ for
  $\partial U$ which is positive on $U$. 

  (1) There is an open neighborhood $V$ of $\partial U$ in $\bar U$
  such that $g_\rho$ defines a smooth pseudo--Riemannian metric on
  $V\cap U$. 

  (2) If $d\rho(\Cal N(\xi,\eta))=0$ for all $\xi,\eta\in\frak X(U)$,
  then $g_\rho$ is Hermitean with respect to $J$.

  (3) If in the setting of (2) $g_\rho$ is admissible, then it is
  c--projectively compact. This is always the case if $J$ is
  integrable.
\end{prop}
\begin{proof}
  Let us use the notation from \ref{2.2a} and put $\th=-d\rho\o
  J$. Since $d\log\rho=\tfrac{d\rho}{\rho}$ our definition implies that
  $\om=d\tfrac1\rho \th=-\tfrac1{\rho^2}d\rho\wedge\th+\tfrac{1}\rho
  d\th$. The definition of $g_\rho$ then directly implies that 
  \begin{equation}
    \label{grho-form}
    g_\rho (\xi,\eta)=\tfrac{-1}{\rho^2}\big(d\rho(\xi)d\rho(\eta)+
\th(\xi)\th(\eta)\big)+\tfrac1\rho d\th(\xi,J\eta). 
  \end{equation}
  Since we have assumed that $\partial U$ is Levi--non--degenerate,
  $(\xi,\eta)\mapsto d\th(\xi,J\eta)$ is non--degenerate on
  $\ker(d\rho)\cap\ker(\th)$ along the boundary. Hence there is a
  neighborhood $V$ of $\partial U$ in $\overline{U}$ such that $d\th$
  is non--degenerate on $\ker(d\rho)\cap\ker(\th)$ on all of $V$. This
  readily implies that $g_\rho$ is non--degenerate on $V\cap U$ so
  (1) is proved.

  (2) The proof of Lemma \ref{lem2.2a} shows that $d\th$ is Hermitean
  on $U$, and then formula \eqref{grho-form} immediately implies that
  $g_\rho$ is Hermitean on $U$, too. 

  (3) In the setting of (2), $\om$ is (up to sign) the fundamental
  two--form associated to the Hermitean metric $g_\rho$. But by
  construction $\om$ is exact and thus closed. In the integrable case,
  this implies that $g_\rho$ is K\"ahler and hence admissible by
  Proposition \ref{prop2.5}. Knowing that $g_\rho$ is admissible, we
  can multiply formula \eqref{grho-form} by $\rho$ and bring the first
  term on the right hand side to the left hand side to see that the
  assumptions of Theorem \ref{thm2.6} are satisfied with $C=-1$ and
  $h(\xi,\eta)=d\th(\xi,J\eta)$.
\end{proof}

\section{C--projective tractors}\label{3}
To prove general results on the existence of asymptotic forms, as in
Theorem \ref{thm2.6}, we use the c--projective version of tractor
calculus. A particularly nice instance of this situation is provided by
reductions of c--projective holonomy to a unitary group, which, in the
integrable case, are related to K\"ahler--Einstein metrics. 

\subsection{The c--projective Schouten tensor}\label{3.1}
As an important ingredient for tractor calculus, we first discuss the
c--projective version of the Schouten tensor (or
Rho--tensor). Similarly to the well known cases of conformal and
projective structures, this is a curvature quantity which is closely
related to the Ricci curvature, but has nicer transformation laws
under a c--projective change. It will be very important for us to work
in a purely real picture (i.e.~without complexifying the tangent
bundle). We will also formulate the results mainly in terms of abstract
indices. 

Let $(N,J)$ be an almost complex manifold, let $\nabla$ be a complex
linear connection on $TN$ and let $R$ be its curvature tensor. In
abstract index notation, the almost complex structure is denote by
$J^a_b$ via $J(\xi)^a=J^a_i\xi^i$, while we write the curvature tensor
as $R_{ab}{}^c{}_d$ via
$R(\xi,\eta)(\ze)^c=R_{ij}{}^c{}_k\xi^i\eta^j\ze^k$. From the
construction it follows readily that $R$ has values in complex linear
maps, i.e.~$R_{ab}{}^c{}_iJ^i_d=R_{ab}{}^i{}_dJ^c_i$. The
\textit{Ricci tensor} $\Ric=\Ric_{ab}$ of $\nabla$ is then defined as
usual by $\Ric_{ab}:=R_{ia}{}^i{}_b$. Observe that $\Ric$ is not
symmetric in general. Finally, the \textit{c--projective Schouten
  tensor} $\Rho=\Rho_{ab}$ of $\nabla$ is defined by 
\begin{equation}
  \label{Rhodef}
  \Rho_{ab}:=\tfrac1{2(m+1)}\left(\Ric_{ab}+
\tfrac1{m-1}\big(\Ric_{(ab)}-J_a^iJ_b^j\Ric_{(ij)}\big)\right),
\end{equation}
where the real dimension of $N$ is $2m$. Observe that this definition
immediately implies that $\Rho$ is symmetric provided that $\Ric$ is
symmetric. Moreover, if $\Ric$ is Hermitean,
i.e.~$\Ric_{ab}=J^i_aJ^j_b\Ric_{ij}$, then $\Rho$ is just a constant
multiple of $\Ric$.

Now we can describe a characterization of the Schouten tensor as well
as its transformation law under a c--projective change of connection. 
\begin{prop}\label{prop3.1}
  (1) The Schouten tensor is uniquely characterized by the property
  that defining $W_{ab}{}^c{}_d$ as 
$$
R_{ab}{}^c{}_d-2\big(\de^c_{[a}\Rho_{b]d}-\Rho_{[ab]}\de^c_d-J^i_{[a}\Rho_{b]i}J^c_d
-J^c_{[a}\Rho_{b]i}J^i_d\big),  
$$
  one has $W_{ia}{}^i{}_b=0$. 

  (2) If $\hat\nabla$ is a connection which is c--projectively related
  to $\nabla$ via $\hat\nabla=\nabla+\Up$ as in formula \eqref{cpdef}
  for a one--form $\Up=\Up_a$, then the Schouten tensor
  $\widehat{\Rho}$ of $\hat\nabla$ is given by
\begin{equation}
  \label{Rhotransf}
  \widehat{\Rho}_{ab}=\Rho_{ab}-\nabla_a\Up_b+\Up_a\Up_b-J^i_aJ^j_b\Up_i\Up_j. 
\end{equation}
\end{prop}
\begin{proof}
  Both properties can be established via slightly lengthy but
  straightforward computations. Alternatively, they can be deduced
  from the formulae in Section 2.4 of \cite{CEMN} taking into account
  the slightly different conventions we use.
\end{proof}

\subsection{The tractor bundle of Hermitian forms}\label{3.2} 
Recall from section \ref{2.1} that almost c--projective structures of
real dimension $n=2m$ with a chosen root $\Cal E(1,0)$ of the
canonical bundle are equivalent to regular normal Cartan geometries of
type $(G,P)$. Here $G=SL(n+1,\Bbb C)$ and $P\subset G$ is the
stabilizer of a line in the standard representation $\Bbb C^{n+1}$ of
$G$. We have also noted there that the Cartan geometry can be
equivalently encoded as the standard tractor bundle $\Cal T$, which is
the associated bundle with respect to the restriction of the
representation $\Bbb C^{n+1}$ to $P$ and the standard tractor
connection $\nabla^{\Cal T}$, the linear connection on $\Cal T$
induced by the Cartan connection.

We will mainly be interested in the bundle $\Cal H$ of Hermitian forms
on $\Cal T$ and in its dual $\Cal H^*$, which can be interpreted as
Hermitean forms on $\Cal T^*$. The bundle $\Cal H$ is a real subbundle
of the tensor product of the complex dual of $\Cal T$ and its
conjugate dual. Since the fact that we deal with sections of the real
bundle is of crucial importance for us, we will avoid working with
complexifications. In the real picture, the simplest way to view
Hermitian forms is as symmetric real valued bilinear forms for which
the almost complex structure is orthogonal (and hence also skew
symmetric).

Our first task is to deduce the composition structure of $\Cal H$
induced by the canonical composition structure of $\Cal T$. The line
in $\Bbb C^{n+1}$ stabilized by $P$ gives rise to a complex line
subbundle $\Cal T^1\subset\Cal T$. For a almost c--projective structure on an
almost complex manifold $(N,J)$ it turns out that $\Cal T^1$ can be
naturally identified with the density bundle $\Cal E(-1,0)$, while
$\Cal T/\Cal T^1\cong TN(-1,0)$, and we summarize this composition
structure as $\Cal T\cong \Cal E(-1,0)\rpl TN(-1,0)$.

Now since $\Cal E(-1,0)$ is a complex line bundle, the bundle of
Hermitian bilinear forms on $\Cal E(-1,0)$ is a real line bundle,
which naturally includes into $\Cal E(1,1)$, so this is the real
density bundle $\Cal E(2)$. Given a Hermitian bilinear form on $\Cal
T$, we can restrict it to $\Cal T^1\x\Cal T^1$, thus defining a
surjective bundle map $\Cal H\to\Cal E(2)$, whose kernel defines a
subbundle $\Cal H^1\subset\Cal H$ of real corank one. Taking an
element of $\Cal H^1$ and restricting to $\Cal T\x\Cal T^1$, the
result descends to a bilinear form on $(\Cal T/\Cal T^1)\x\Cal T^1$,
and we define $\Cal H^2\subset\Cal H^1$ to be the subbundle of those
elements for which also this induced form vanishes. Thus we have
obtained a filtration $\Cal H\supset\Cal H^1\supset\Cal H^2$ of the
bundle $\Cal H$. Now we can describe the composition structure of
$\Cal H$.

\begin{prop}\label{prop3.2}
  The composition structure for the canonical filtration $\Cal
  H\supset\Cal H^1\supset\Cal H^2$ of $\Cal H$ has the form
$$
\Cal H\cong \Herm(TN)\otimes\Cal E(2) \rpl T^*N\otimes\Cal E(2)\rpl\Cal
E(2),
$$ where all tensor products are over $\Bbb R$ and $\Herm(TN)$ denotes
the bundle of (real valued) Hermitian forms on the complex vector
bundle $TN$.
\end{prop}
\begin{proof}
  We have already seen that $\Cal H/\Cal H^1\cong\Cal E(2)$ and that
  $\Cal H^1/\Cal H^2$ and $\Cal H^2$ can be identified with the space
  of Hermitian bilinear forms on $\Cal T/\Cal T^1\x\Cal T^1$ and on
  $\Cal T/\Cal T^1\x\Cal T/\Cal T^1$, respectively. Identifying these
  spaces of Hermitean forms is a purely pointwise questions of linear
  algebra.

  Thus let us assume that $V$ and $\ell$ are complex vector spaces of
  complex dimension $m$ and $1$, respectively. To complete the proof
  we have to describe Hermitian bilinear forms on $(V\otimes_{\Bbb
    C}\ell)\x\ell$ and $(V\otimes_{\Bbb C}\ell)\x (V\otimes_{\Bbb
    C}\ell)$ in terms of the space $H_\ell$ of Hermitian forms on
  $\ell$.

  Observe first that the complex tensor product $V\otimes_{\Bbb
    C}\ell$ is the quotient of the real tensor product $V\otimes\ell$
  by the linear subspace spanned by all elements of the form $v\otimes
  a+iv\otimes ia$ for $v\in V$ and $a\in\ell$. Given a (real) linear
  map $\ph:V\to H_\ell$ we consider the map $V\x \ell\to L(\ell,\Bbb
  R)$ mapping $(v,a)$ to $\ph(v)(a,\_)-\ph(iv)(ia,\_)$. This is
  evidently bilinear over the reals, so it factorizes to
  $V\otimes\ell$ and then kills all elements of the form $v\otimes
  a+iv\otimes ia$, so it induces a bilinear map $(V\otimes_{\Bbb
    C}\ell)\x\ell\to\Bbb R$. Now the complex structure on
  $V\otimes_{\Bbb C}\ell$ is given by multiplication by $i$ on either
  component. So to verify that our map is indeed Hermitian, it
  suffices to compare the values of the corresponding trilinear map
  $V\x\ell\x\ell\to\Bbb R$ on $(v,a,b)$ and $(v,ia,ib)$. But for the
  latter element we get
$$
\ph(v)(ia,ib)-\ph(iv)(-a,ib)=\ph(v)(a,b)-\ph(iv)(ia,b). 
$$ 
Thus we have constructed a linear map from $L_{\Bbb R}(V,H_\ell)$ to
the space of Hermitian bilinear forms on $(V\otimes_{\Bbb
  C}\ell)\x\ell$. If $\ph$ is in the kernel of this map, we in
particular get $0=\ph(v)(a,a)-\ph(iv)(ia,a)$ for all $v\in V$ and
$a\in\ell$. But any element of $H_\ell$ vanishes on $(ia,a)$ and
non--zero elements are non--zero on $(a,a)$ provided that $a\neq
0$. Hence we see that $\ph(v)=0$ for all $v\in V$. Thus our map is
injective and hence a linear isomorphism by a dimension count.

Next assume that $\ps:V\x V\to H_\ell$ is symmetric and bilinear (over
$\Bbb R$) and satisfies $\ps(iv,iw)=\ps(v,w)$ for all $v,w\in V$. Then
similar to the above we can consider the $4$--linear map $V\x \ell\x
V\x\ell$ defined by 
$$
(v,a,w,b)\mapsto \ps(v,w)(a,b)-\ps(iv,w)(ia,b).
$$
One verifies step by step that this map induces a Hermitean
bilinear map $(V\otimes_{\Bbb C}\ell)\x (V\otimes_{\Bbb C}\ell)\to\Bbb
R$ and that this gives rise to the required isomorphism.
\end{proof}

Using this result, we can now give a description of $\Cal H$ in a
purely real picture. We start by discussing the explicit description
of $\Cal T$ following \cite{CEMN} and the general results in chapter 5
of \cite{book}.

A choice of a connection $\nabla$ in the c--projective class gives
rise to a splitting $\Cal T\cong TN\otimes_{\Bbb C}\Cal
E(-1,0)\oplus\Cal E(-1,0)$ of the canonical filtration. We will write
elements of this bundle in such a splitting as vectors of the form
$\binom{\xi\otimes_{\Bbb C}\rho}{\nu}$ with $\nu$, $\rho$ in $\Cal
E(-1,0)$ and $\xi\in TM$. In this picture, the behavior under a change
of connection is easy to describe. Changing from $\nabla$ to
$\hat\nabla=\nabla+\Up$ as in \ref{2.1}, the change is given by
\begin{equation}\label{eq:t-trans}
\widehat{\begin{pmatrix} \xi\otimes_{\Bbb C}\rho \\ \nu\end{pmatrix}}
=\begin{pmatrix} \xi\otimes_{\Bbb C}\rho
\\ \nu-\Up(\xi)\rho+\Up(J\xi)i\rho \end{pmatrix}.
\end{equation}
Here $\Up$ is viewed as real valued and as in the proof of Proposition
\ref{prop3.2} one checks that this expression really makes sense for a
tensor product over $\Bbb C$. 

Now in view of Proposition \ref{prop3.2}, having chosen a connection
$\nabla$ in the c--projective class, we represent elements of $\Cal H$
as triples (or vectors) of the form $(\tau,\ph,\ps)$ were $\tau$ is an
element of the space $\Cal E(2)$ of Hermitean forms on $\Cal E(-1,0)$,
$\ph$ is a linear map $TN\to\Cal E(2)$, and $\ps$ is
a Hermitean form $TN\x TN\to\Cal E(2)$. The action of such a
triple is then defined as 
\begin{equation}\label{eq:pair-def}
\begin{aligned}
\begin{pmatrix}\tau \\ \ph\\ \ps 
\end{pmatrix}&\left(\begin{pmatrix}\xi_1\otimes_{\Bbb C}\rho_1
  \\ \nu_1\end{pmatrix}, \begin{pmatrix}\xi_2\otimes_{\Bbb C}\rho_2
  \\ \nu_2\end{pmatrix}\right):= \\
\tau(\nu_1,\nu_2)+&\ph(\xi_1)(\rho_1,\nu_2)-\ph(J\xi_1)(i\rho_1,\nu_2)+
\ph(\xi_2)(\rho_2,\nu_1)\\
-&\ph(J\xi_2)(i\rho_2,\nu_1)+\ps(\xi_1,\xi_2)(\rho_1,\rho_2)-
\ps(J\xi_1,\xi_2)(i\rho_1,\rho_2).  
\end{aligned}
\end{equation}
Changing to a different connection $\hat\nabla=\nabla+\Up$ as in
\ref{2.1} we know the change of splittings on $\Cal T$, and of course
the change of splitting on $\Cal H$ is determined by 
$$
\widehat{\begin{pmatrix}\tau \\ \ph\\ \ps 
\end{pmatrix}}\left(\widehat{\begin{pmatrix}\xi_1\otimes_{\Bbb C}\rho_1
  \\ \nu_1\end{pmatrix}}, \widehat{\begin{pmatrix}\xi_2\otimes_{\Bbb C}\rho_2
  \\ \nu_2\end{pmatrix}}\right)=\begin{pmatrix}\tau \\ \ph\\ \ps 
\end{pmatrix}\left(\begin{pmatrix}\xi_1\otimes_{\Bbb C}\rho_1
  \\ \nu_1\end{pmatrix}, \begin{pmatrix}\xi_2\otimes_{\Bbb C}\rho_2
  \\ \nu_2\end{pmatrix}\right). 
$$
A lengthy but straightforward computation then shows that (as
expected) $\widehat{\tau}=\tau$, while
$\widehat{\ph}(\xi)=\ph(\xi)+\Up(\xi)\tau$ and 
\begin{align*}
\widehat{\ps}(\xi_1,\xi_2)=\ps(\xi_1,\xi_2)+&\Up(\xi_1)\ph(\xi_2)+
\Up(J\xi_1)\ph(J\xi_2)+\Up(\xi_2)\ph(\xi_1)+\Up(J\xi_2)\ph(J\xi_1)\\
+&\bigg(\Up(\xi_1)\Up(\xi_2)+\Up(J\xi_1)\Up(J\xi_2)\bigg)\tau.
\end{align*}

This shows that we can nicely formulate things in (real) abstract
index notation. Following the usual conventions, we write $\ph=\ph_a$
for a section of $\Cal E_a(2)$ and $\ps=\ps_{ab}$ a section of $\Cal
E_{ab}(2)$ such that $\ph_{ab}=\ph_{cd}J_a^cJ^d_b$. Then the
transformation determined by $\Up=\Up_a$ (a section of $\Cal E_a$) can
be expressed as
\begin{equation}\label{h-trans}
\widehat{\begin{pmatrix}\tau \\ \ph_a\\ \ps_{bc} 
\end{pmatrix}}=\begin{pmatrix}\tau \\ \ph_a+\Up_a\tau\\ 
\ps_{bc}+(\delta^i_b\delta^j_c+J^i_bJ^j_c)(\Up_i\ph_j+\Up_j\ph_i+\Up_i\Up_j\tau)  
\end{pmatrix}. 
\end{equation}

\subsection{The tractor connection on $\Cal H$}\label{3.3} 
The relation between $\Cal H$ and $\Cal T$ introduced in \ref{3.2} can
be also used to compute the tractor connection $\nabla^{\Cal H}$ on
$\Cal H$. The first step towards this to describe the standard tractor
connection on $\Cal T$. This can be taken from \cite{CEMN} or directly
deduced from the principles in Section 5.2 \cite{book} (taking into
account different sign conventions for the rho--tensor). In the
notation from \ref{3.2} for the standard tractor bundle, this can be
explicitly described as 
\begin{equation}\label{eq:tconn-def}
\nabla_{\eta}^{\Cal T}\begin{pmatrix} \xi\otimes_{\Bbb C}\rho \\ \nu
\end{pmatrix}=\begin{pmatrix} (\nabla_\eta\xi)\otimes_{\Bbb
  C}\rho+\xi\otimes_{\Bbb C}\nabla_\eta\rho+\eta\otimes_{\Bbb C}\nu \\
\nabla_\eta \nu-\Rho(\eta,\xi)\rho+\Rho(\eta,J\xi)i\rho
\end{pmatrix}, 
\end{equation}
where $\Rho$ is the Schouten tensor defined in \ref{3.1}.  As before,
we view $\Rho$ as an ordinary $\binom02$--tensor field on $N$,
i.e.~$\Rho(x)$ is a real valued bilinear form on $T_xN$. Nonetheless,
the combination occurring in the second row is well defined on complex
tensor products.

Now we can use the paring defined in \eqref{eq:pair-def} to compute
the tractor connection $\nabla^{\Cal H}$ on $\Cal H$ via
$$
(\nabla^{\Cal
  H}_\eta\Ph)(s_1,s_1)=\eta\cdot(\Ph(s_1,s_2))-\Ph(\nabla_\eta^{\Cal
  T}s_1,s_2)- \Ph(s_1,\nabla_\eta^{\Cal T}s_2)
$$ 
for $\Ph\in\Ga(\Cal H)$ and $s_1,s_2\in\Ga(\Cal T)$. Computing this in
the splitting determined by a choice of connection $\nabla$ from the
c--projective class is again straightforward but a bit tedious, and
the result is
$$
\nabla_\eta^{\Cal H} \begin{pmatrix}\tau\\ \ph\\ \ps \end{pmatrix}=
\begin{pmatrix} \nabla_\eta\tau - 2\ph(\eta) \\ 
\nabla_\eta\ph+\Rho(\eta,\_)\tau-\ps(\eta,\_) \\
\nabla_\eta\ps+ A(\eta)\end{pmatrix}, 
$$
where $A$ is defined by  
$$
A(\eta)(\xi_1,\xi_2)=\Rho(\eta,\xi_1)\ph(\xi_2)+\Rho(\eta,\xi_2)\ph(\xi_1)+
\Rho(\eta,J\xi_1)\ph(J\xi_2)+\Rho(\eta,J\xi_2)\ph(J\xi_1).
$$ But this shows that we can express the tractor connection on $\Cal
H$ nicely in the the abstract index notation introduced in Section
\ref{3.2} as
\begin{equation}\label{eq:hconn-def} 
\nabla^{\Cal H}_a\begin{pmatrix} \tau \\ \ph_b
\\ \ps_{cd} \end{pmatrix} = \begin{pmatrix} \nabla_a\tau-2\ph_a \\ 
\nabla_a\ph_b+\Rho_{ab}\tau-\ps_{ab}
\\ \nabla_a\ps_{cd}+\Rho_{ac}\ph_d+\Rho_{ad}\ph_c+\Rho_{ai}J^i_c\ph_jJ^j_d+ 
\Rho_{ai}J^i_d\ph_jJ^j_c\end{pmatrix}. 
\end{equation}

\begin{remark}\label{rem3.3}
  Formally, there is no need to use the relation to parabolic
  geometries in order to obtain the elements of tractor calculus we
  need here. Following the spirit of \cite{BEG}, one could define the
  standard tractor bundle $\Cal T$ by requiring that any choice of a
  connection $\nabla$ from the c--projective class defines an
  isomorphism $\Cal T\cong (TN\otimes_{\Bbb C}\Cal E(-1,0))\oplus\Cal
  E(-1,0)$ and that the change of isomorphism for a c--projectively
  related connection $\hat\nabla$ is described by
  \eqref{eq:t-trans}. Using the transformation of the Schouten tensor
  from Proposition \ref{prop3.1} as well as formula
  \eqref{eq:cdens-trans}, one then verifies directly that the result
  of formula \eqref{eq:tconn-def} is the same for all connections in
  the c--projective class. Hence one can use it to \textit{define} a
  linear connection $\nabla^{\Cal T}$ on $\Cal T$, and then clearly
  the pair $(\Cal T,\nabla^{\Cal T})$ is canonically associated to the
  almost c--projective structure (and the choice of $\Cal E(1,0)$). This
  basically recovers all the input needed for the further
  developments.
\end{remark}

\subsection{The metricity bundle and its connection}\label{3.4} 
The \textit{metricity bundle} is the dual bundle $\Cal H^*$ to $\Cal
H$, which of course inherits a connection $\nabla^{\Cal H^*}$ from
$\nabla^{\Cal H}$ (and hence from the standard tractor connection). It
can be viewed as the bundle of Hermitean forms on $\Cal T^*$, but this
point of view is not important at this stage. We can recover the
description of the bundle and the connection just by dualizing $\Cal
H$.

From the duality, we can readily derive a canonical composition
structure of $\Cal H^*$. The annihilators of $\Cal H^1\supset\Cal H^2$
define a natural real line subbundle contained in a natural subbundle
of real rank $2m+1$. The corresponding composition series has the form 
$$
\Cal H^*=\Cal E(-2)\rpl TN\otimes\Cal E(-2)\rpl \Herm(T^*N)\otimes\Cal
E(-2).
$$ Choosing a connection $\nabla$ in the c--projective class, one
obtains a splitting of this filtration, i.e. an isomorphism $\Cal
H^*\cong\Herm(T^*N)\otimes\Cal E(-2)\oplus TN\otimes\Cal
E(-2)\oplus\Cal E(-2)$. We denote corresponding elements as column
vectors $(\si^{ab},\mu^c,\nu)^t$. We fix the pairing of this with
elements in the splitting of $\Cal H$ determined by $\nabla$ from
Section \ref{3.2} by
\begin{equation}
  \label{eq:dual-def}
\langle(\si^{ab},\mu^c,\nu)^t,(\tau,\ph_a,\ps_{cd})^t\rangle:=
\tau\nu+\ph_i\mu^i+\tfrac12\ps_{ij}\si^{ij}.  
\end{equation}
In most of what follows, the precise form of this pairing is not
important, but at one stage, we have to compute the inverse of a
non--degenerate section of $\Cal H^*$. There the factors are
important, so let us briefly explain this choice. To determine the 
inverse, one converts a Hermitean form on $\Cal T$ into a linear map
$\Cal T\to\Cal T^*$ and then converts the inverse of this map into a
bilinear form on $\Cal T^*$. The natural idea to choose a pairing thus
is to convert the Hermitean forms into linear maps, compose them to
obtain an endomorphism of $\Cal T$ and define the pairing to be its
trace. However, we have to be careful about the difference between
real and complex traces in this setting. Taking the pairing of two
elements concentrated in the one--dimensional slots as $\tau\nu$ shows
that we work with the real part of the complex trace, since these act
on a complex line bundle. But then the factor $\ps_{ai}\si^{ib}$ is
the real matrix representing a complex linear map, so the real part of
its complex trace is half the real trace, which motivates putting
$\tfrac12\ps_{ij}\si^{ij}$ into the pairing. Since we will need
inverses only in scales for which the middle slots are trivial, this
is all the information we need. The choice of a factor one for
$\ph_i\mu^i$ is motivated by the fact that, while we take a real
trace, $\ph_b$ and $\mu_c$ each represent one row and one column in a
Hermitian matrix, so usually, they would be counted with a factor
two. 

\begin{prop}\label{prop3.4}
  Changing from $\nabla$ to a c--projectively related connection
  $\hat\nabla=\nabla+\Up$ as in formula \eqref{cpdef} with a one--form
  $\Up=\Up_a$, the corresponding splittings of the composition
  structure are related by 
  \begin{equation}
    \label{eq:met-trans}
    \widehat{\begin{pmatrix} \si^{ab} \\ \mu^c \\ \nu \end{pmatrix}}= 
    \begin{pmatrix} \si^{ab} \\ \mu^c-2\Up_i\si^{ic} \\
      \nu-\Up_i\mu^i+\Up_i\Up_j\si^{ij}\end{pmatrix}
  \end{equation}
  The connection $\nabla^{\Cal H^*}$ dual to $\nabla^{\Cal H}$ is, in
  the splitting corresponding to $\nabla$, given by
\begin{equation}
  \label{eq:met-conn}
  \nabla_a^{\Cal H^*}\begin{pmatrix} \si^{bc} \\ \mu^d \\
    \nu \end{pmatrix}=\begin{pmatrix}
    \nabla_a\si^{bc}+\delta^{(b}_a\mu^{c)}+J_a^{(b}J_i^{c)}\mu^i \\ 
    \nabla_a\mu^d-2\si^{di}\Rho_{ai}+2\nu\delta^d_a \\ 
    \nabla_a\nu-\mu^i\Rho_{ai}\end{pmatrix}
\end{equation}
\end{prop}
\begin{proof}
  The change of splittings is determined by the fact that the pairing
  in formula \eqref{eq:dual-def} has to be the same in both
  splittings. The claimed formula then follows from \eqref{h-trans} by
  a simple direct computation. Likewise, the dual connections by
  definition satisfy 
$$
\nabla_a\left\langle\begin{pmatrix}\tau \\ \ph_b\\ \ps_{cd} 
  \end{pmatrix},\begin{pmatrix} \si^{ij} \\ \mu^k \\
    \nu \end{pmatrix}\right\rangle=
\left\langle\nabla^{\Cal H}_a\begin{pmatrix}\tau \\ \ph_b\\ \ps_{cd} 
  \end{pmatrix},\begin{pmatrix} \si^{ij} \\ \mu^k \\
    \nu \end{pmatrix}\right\rangle+\left\langle\begin{pmatrix}\tau \\
    \ph_b\\ \ps_{cd}
  \end{pmatrix},\nabla^{\Cal H^*}_a\begin{pmatrix} \si^{ij} \\ \mu^k \\
    \nu \end{pmatrix}\right\rangle,
$$
and using \eqref{eq:hconn-def}, the claimed formula follows by a
straightforward computation.
\end{proof}

\subsection{The metricity equation}\label{3.5} 
Via the general machinery of BGG--sequences, see \cite{CSS-BGG} and
\cite{Calderbank-Diemer}, any tractor bundle gives rise to a sequence
of invariant differential operators. The first of these operators is
always overdetermined. For the case of $\Cal H^*$, the corresponding
facts can also be verified by rather simple direct computations. 

The bundle $T^*N\otimes \Herm(T^*N)$ can be naturally decomposed a
trace part and a trace--free part. In fact the maps showing up in the
top line of \eqref{eq:met-conn} represent the inclusion of the trace
part. In abstract index notation, a section of $T^*N\otimes \Herm(T^*N)$
has the form $\ps_a{}^{bc}$ and this is called trace--free if
$\ps_i{}^{ia}=0$. (Note that by symmetry, there is just one possible
trace.) On the other hand, for a section $\mu^a$ of $TN$, one
immediately verifies that
$\delta_a^{(b}\mu^{c)}+J^{(b}_aJ^{c)}_i\mu^i$ is a section of
$T^*N\otimes \Herm(T^*N)$. The trace of this section is $m\mu^a$, where
$\dim_{\Bbb R}(N)=2m$. Hence we see that we can uniquely decompose any
$\ps_a{}^{bc}$ into a trace part and a tracefree part, with the latter
being given by 
$$
\tfp(\ps_a{}^{bc}):=\ps_a{}^{bc}-\tfrac1m\left(\delta_a^{(b}\ps_i{}^{c)i}+
J^{(b}_aJ^{c)}_i\ps_j{}^{ij}\right).
$$ Of course, this works in the same way if we twist by a real density
bundle.

\begin{prop}\label{prop3.5}
(i) For any section $\si^{bc}$ of $\Herm(T^*N)\otimes\Cal E(-2)$, the
  tracefree part $\tfp(\nabla_a\si^{bc})$ is independent of the choice
  of connection $\nabla$ in the c--projective class. Mapping
  $\si^{bc}$ to $\tfp(\nabla_a\si^{bc})$ is an invariant differential
  operator.

(ii) There is an invariant differential operator
  $L:\Ga(\Herm(T^*N)\otimes\Cal E(-2))\to\Ga(\Cal H^*)$ such that 
\begin{itemize}
\item the top component of $L(\si^{ab})$ equals $\si^{ab}$
\item $\si^{bc}$ satisfies $\tfp(\nabla_a\si^{bc})=0$ if and only if
  $\nabla_a^{\Cal H^*}L(\si)$ has vanishing top component (and hence
  is a one--form with values in annihilator of $\Cal H^2\subset\Cal
  H$).
\end{itemize}

(iii) There is a natural linear connection $\nabla^p$ on the bundle
$\Cal H^*$ such that for a section $\si^{ab}$,
$\tfp(\nabla_a\si^{bc})=0$ is equivalent to $\nabla^p_a
L(\si^{bc})=0$.
\end{prop}
\begin{proof}
  In terms of the theory of BGG sequences, $L$ is the BGG--splitting
  operator and $\si^{bc}\mapsto \tfp(\nabla_a\si^{bc})$ is the
  corresponding first BGG operator, which implies (i) and the first
  part of (ii). The second part of (ii) is an explicit version of the
  fact that $\si$ is a solution of the first BGG operator if and
  only if $\nabla^{\Cal H^*}(\si)$ is a section of the subbundle
  $\im(\partial^*)$, where $\partial^*$ is the
  Kostant--codifferential. Part (iii) follows from the general
  construction of prolongation connections in \cite{HSSS}. 

  There also is a direct proof of this result along the following
  lines. Using equation \eqref{eq:met-trans}, one immediately verifies
  that mapping $(\si_a{}^{bc},\mu_a{}^b,\nu_a)^t\in T^*N\otimes\Cal
  H^*$ to $(0,\si_i{}^{ia},\frac12\mu_i{}^i)^t\in\Cal H^*$ is
  independent of the choice of splitting, thus defining a natural
  bundle map. Given $\si^{ab}$ one verifies that, in a given
  splitting, there is a unique section $L(\si^{ab})\in\Ga(\Cal H^*)$
  with top entry equal to $\si^{ab}$ such that $\nabla^{\Cal
    H^*}_aL(\si^{bc})$ lies in the kernel of this bundle
  map. Explicitly, this simply means that the top two components of
  $\nabla^{\Cal H^*}_aL(\si^{bc})$ are trace--free. From this a direct
  computation shows that, in the splitting determined by $\nabla$, one
  gets
  \begin{equation}
    \label{eq:met-split}
    L(\si^{ab})=\begin{pmatrix} \si^{ab}\\ -\tfrac1m\nabla_i\si^{ic}
    \\ \tfrac1{4m^2}\nabla_i\nabla_j\si^{ij}+
    \tfrac{1}{2m}\si^{ij}\Rho_{ij}
    \end{pmatrix}
  \end{equation}
  Since the characterization of $L(\si^{bc})$ is independent of all
  choices we obtain an invariant operator $L$. A part of the above
  computation shows that the top component of $\nabla_a^{\Cal
    H^*}L(\si^{bc})$ equals $\tfp(\nabla_a\si^{bc})$, which implies
  (i) and (ii). A direct construction for $\nabla^p$ on the
  complexification of $\Cal H^*$ can be found in Theorem 4.6 of
  \cite{CEMN}.
\end{proof}

\subsection{Remarks on BGG equations and holonomy
  reductions}\label{3.6} 

The only result on the metricity equation we will need in what follows
is that any admissible Hermitean pseudo--Riemannian metric, whose
canonical connection lies in the c--projective class gives rise to a
solution to the metricity equation. Given $g_{ab}$, one simply takes
the inverse metric $g^{ab}$ and multiplies is by an appropriate power
of the volume density of $g$ to obtain a section $\si^{ab}$ of the
right weight. This is even parallel for the canonical connection, so
by projective invariant $\tfp(\tilde\nabla_a\si^{bc})=0$ for any
connection $\tilde\nabla_a$ in the c--projective class. It is not
difficult to show that any solution $\si^{ab}$ to the metricity
equation which is non--degenerate as a bilinear form is obtained in
this way, compare with Section 4.3 of \cite{CEMN}. 

As we have noted in the proof of Proposition \ref{prop3.5}, the
metricity equation is a so--called first BGG equation. It is known in
general for such equation that there is a subspace of so--called
\textit{normal solutions}, which are characterized by the fact that
$L(\si^{ab})$ is parallel for the tractor connection $\nabla_a^{\Cal
  H^*}$ and these exhaust all parallel sections of $\Cal H^*$. We can
now easily describe these normal solutions:

As noted above, any solution $\si^{ab}$ determines a connection
$\nabla_a$ in the c--projective class such that $\nabla_a\si^{bc}=0$.
In the scale determined by this connection, the splitting operator
simplifies to
$L(\si^{ab})=(\si^{ab},0,\tfrac1{2m}\si^{ij}\Rho_{ij})^t$. This also
shows that in this splitting 
$$
\nabla_a^{\Cal H^*}L(\si^{bc})=\begin{pmatrix} 0 \\
\tfrac1{m}\si^{ij}\Rho_{ij}\delta^b_a-2\si^{bi}\Rho_{ai}\\
\tfrac1{2m}\nabla_a\si^{ij}\Rho_{ij}
\end{pmatrix}. 
$$ Non--degeneracy of $\si^{ab}$ implies that there is an inverse
$\si_{ab}$ which is also Hermitean. Then vanishing of the middle slot
is equivalent to the fact that $\Rho_{ab}$ is proportional to
$\si_{ab}$. This is equivalent to the fact that $\si$ comes from a
Hermitean Einstein metric, see also Proposition 4.8 in \cite{CEMN}. 
This nicely complements the result in \cite{CGM} for the projective
metricity equation. 

The parallel section of a tractor bundle corresponding to a normal
solution of the metricity equation defines a holonomy reduction of the
canonical Cartan connection associated to the almost c--projective structure
to a special unitary group, see \cite{hol-red} for the general theory
of these reductions. If the parallel section $L(\si^{ab})$ of $\Cal
H^*$ is non--degenerate as a bilinear form on $\Cal T^*$ (which
essentially means that the metric is Einstein but not Ricci flat),
then its inverse defines a parallel section $\Cal H$. In the latter
interpretation, such holonomy reductions are discussed in Section 3.3
of \cite{hol-red}. Also in this dual picture, there is a parallel
story with a first BGG equation and normal solutions of this equation
as follows, compare with Sections 4.6 and 4.7 of \cite{CEMN}.

To a section $\tau$ of $\Cal E(2)$ one naturally associates a section
$L(\tau)$ of $\Cal H$ as follows. One requires $L(\tau)$ to have
$\tau$ in the top slot, and also requires that $\nabla_a^{\Cal H}L(\tau)$ has
vanishing top slot while its middle slot has vanishing symmetric
Hermitean part. 
These two conditions are immediately seen to be independent of the
choice of splitting and, using this characterization, one directly computes that, in the splitting
corresponding to $\tilde\nabla_a$, one gets
$$
L(\tau)=\begin{pmatrix} \tau \\ \tfrac12\tilde\nabla_a\tau
\\ \tfrac12(\delta^i_{(b}\delta^j_{c)}+J^i_{(b}J^j_{c)})
(\tfrac12\tilde\nabla_i\tilde\nabla_j\tau+\tilde\Rho_{ij}\tau) 
\end{pmatrix},
$$ 
where $\tilde\Rho_{ab}$ is the c--projective Schouten tensor of
$\tilde\nabla_a$.  This then implies that mapping $\tau$ to the
anti--Hermitean symmetric part of
$\tilde\nabla_a\tilde\nabla_b\tau+2\tilde\Rho_{ab}\tau$ defines an
invariant differential operator, which is the first BGG operator
associated to $\Cal H$.

If $\tau$ lies in the kernel of this operator, then it is easy to see
that $\tau$ is non--vanishing on a dense open subset. On this subset,
there is a unique connection $\nabla_a$ in the c--projective class for
which $\nabla_a\tau=0$. Then the c--projective Schouten tensor
$\Rho_{ab}$ must be symmetric and since $\tau$ solves the BGG
equation, $\Rho_{ab}$ has to be Hermitean. Conversely, it is easy to
see that any connection in the c--projective class with symmetric and
Hermitian projective Schouten tensor locally gives rise to a solution
of the first BGG equation.

In the scale determined by $\nabla_a$, we then have
$L(\tau)=(\tau,0,\tau\Rho_{ab})^t$ and this readily shows that in this
splitting $\nabla_a^{\Cal
  H}L(\tau)=(0,0,\tau\nabla_a\Rho_{bc})$. Hence normal solutions are
exactly those, for which $\Rho_{bc}$ is in addition parallel for
$\nabla_a$. If $\Rho_{bc}$ is non--degenerate as a bilinear form, it
defines a Hermitean metric, which must be admissible since it is
preserved by the minimal complex connection $\nabla_a$, which
therefore has to be its canonical connection. Finally, since the
c--projective Schouten tensor is a multiple of the Ricci curvature of
$\nabla_a$, this metric has to be Einstein.

\section{c--projectively  compact metrics}\label{4} 

\subsection{The c--projective interpretation of scalar
  curvature}\label{4.1} Let $(N,J)$ be an almost complex manifold and
let $g$ be an admissible pseudo--Riemannian metric which is Hermitean
with respect to $J$. We show that $g$ determines a section of the
tractor bundle $\Cal H^*$ over $N$, which in turn leads to a
c--projective interpretation of the scalar curvature of the canonical
connection of $g$.

On the one hand, for $\dim_{\Bbb R}(N)=2m$ the bundle of real volume
densities was defined in \ref{2.1} to be $\Cal E(-2m-2)$. As a
pseudo--Riemannian metric, $g$ determines a volume density
$\vol_g\in\Ga(\Cal E(-2m-2))$, which is parallel for any linear
connection which preserves $g$ and thus, in particular, is nowhere
vanishing. One can take any root of this density to obtain a nowhere
vanishing section of any real density bundle which is parallel for the
canonical connection $\nabla$ of $g$. We will mainly need
$\tau:=\vol_g^{-1/(m+1)}\in\Ga(\Cal E(2))$ and its inverse
$\tau^{-1}\in\Ga(\Cal E(-2))$.

On the other hand, a smooth section $H$ of $\Cal H^*$ defines a
Hermitean form on the cotractor bundle $\Cal T^*$. To such a form, one
may associate a well--defined determinant, which is non--zero if and
only if the form is non--degenerate, as follows. Take the complex
valued Hermitean extension $\tilde H$ of $H$ and for a complex local
frame $\{v_1,\dots,v_{m+1}\}$ for $\Cal T^*$ consider $\det(\tilde
H(v_i,v_j))$, which is real since the matrix is Hermitean. Changing
the local frame corresponds to a smooth function $A$ with values in
$GL(m+1,\Bbb C)$ and the corresponding change of matrix is given by
$A(\tilde H(v_i,v_j))A^*$, so the determinant changes by
multiplication with $|\det_{\Bbb C}(A)|^2=\det_{\Bbb R}(A)$. This
shows that 
$$
v_1\wedge Jv_1\wedge\dots\wedge v_{m+1}\wedge Jv_{m+1}\mapsto 
\det(\tilde H(v_i,v_j))
$$ 
induces a well defined field of linear functionals on the top
exterior power $\La^{2m+2}_{\Bbb R}\Cal T^*$. Thus we obtain a well
defined section $\det(H)\in\Ga(\La^{2m+2}_{\Bbb R}\Cal T)$. Since
$\Cal T$ is induced by a principal bundle with structure group
$SL(m+1,\Bbb C)$ endowed with a canonical connection, the bundle
$\La^{2m+2}_{\Bbb R}\Cal T$ is naturally trivial. Hence up to an
overall, non--zero constant factor, we can view $\det(H)$ as a smooth
function on $N$. It is also easy to see that $\det(H)$ can be computed
as the square root of the determinant of the real Gram matrix of $H$. 

\begin{prop}\label{prop4.1}
  Let $(N,J)$ be an almost complex manifold, $g=g_{ab}$ a
  pseudo--Riemannian metric on $N$, which is Hermitean for $J$ and
  admissible. Put $\tau:=\vol_g^{-1/(m+1)}\in\Ga(\Cal E(2))$ and let
  $g^{ab}$ be the inverse metric for $g$. Then for the almost c--projective
  structure on $N$ determined by the canonical connection $\nabla$ of
  $g$, we have 

  (i) The section $\si^{ab}:=\tau^{-1}g^{ab}$ of $\Herm(T^*N)\otimes\Cal
  E(-2)$ is a solution of the metricity equation from part (i) of
  Proposition \ref{prop3.5}.

  (ii) The image $H:=L(\si^{ab})\in\Ga(\Cal H^*)$ of $\si^{ab}$ under
  the splitting operator has the property that, up to an overall
  non--zero constant, $\det(H)$ coincides with the scalar curvature
  $S:=g^{ij}\Ric_{ij}$ of $\nabla$.
\end{prop}
\begin{proof}
(i) For the canonical connection $\nabla$, we of course have
  $\nabla_ag^{bc}=0$ and $\nabla_a\tau^{-1}=0$. Hence
  $\nabla_a\si^{bc}=0$ and the claim follows from c--projective
  invariance of the metricity equation.

\smallskip

(ii) We can compute in the splitting determined by $\nabla$, in which
$\nabla_a\si^{bc}=0$. In this splitting, formula \eqref{eq:met-split}
simplifies to $L(\si^{ab})=(\si^{ab},0,\frac1{2m}
\si^{ij}\Rho_{ij})^t$, and the last component equals
$\frac1{2m}\tau^{-1}g^{ij}\Rho_{ij}$. Now since $g_{ab}$ is Hermitean,
so is $g^{ab}$, and using this the defining equation \eqref{Rhodef}
for the c--projective Schouten tensor shows that
$g^{ij}\Rho_{ij}=\tfrac1{2(m+1)}g^{ij}\Ric_{ij}$. Hence we end up with
$H=(\tau^{-1}g^{ab},0,C\tau^{-1}S)^t$ for some non--zero constant $C$
and in the splitting corresponding to $\nabla$.

To compute the determinant $\det(H)$, we need to interpret the
notation as triples correctly. The choice of $\nabla$ induces a
splitting of the cotractor bundle $\Cal T^*$ into a direct sum of
complex bundle of rank $m$ and a complex line bundle. The fact that
the middle component of the triple is zero says that this splitting is
orthogonal for $H$. The real function $C\tau^{-1}S$ describes the
action on the complex line bundle, whereas $\tau^{-1}g^{ab}$ is the
real matrix associated to the restriction of $h$ to the complex rank
$n$ bundle. From the discussion of determinants above, we conclude
that, up to a non--zero overall constant (which includes $C$),
$$
\det(H)=\tau^{-1}S\sqrt{\det(\tau^{-1}g^{ab})}=
\tau^{-1}S\sqrt{\tau^{-2m}\det(g^{ab})}=S\tau^{-m-1}(\vol_g)^{-1}=S. 
$$
\end{proof}

This has immediate consequences for our setting of a manifold with
boundary and an admissible Hermitean metric in the interior. Namely,
if we assume that the almost c--projective structure determined by the
canonical connection $\nabla$ admits a smooth extension to the
boundary, then we can use this structure to prove extendibility of
several quantities associated to $\nabla$. 

\begin{cor}\label{cor4.1}
  Let $\barm$ be a smooth manifold of real dimension $2m$ with
  boundary $\partial M$ and interior $M$, let $J$ and $g$ be an almost
  complex structure and a pseudo--Riemannian metric on $M$, such that
  $g$ is Hermitean with respect to $J$ and admissible. Let $g^{ab}$ be
  the inverse of $g$, $\nabla$ its canonical connection, $\vol_g$ its
  volume density and put $\tau:=(\vol_g)^{-1/(m+1)}$. Suppose further
  that the c--projective structure determined by $\nabla$ admits a
  smooth extension to all of $\barm$, so that the tractor bundle $\Cal
  H^*$ is defined on $\barm$. 

  Then the sections
  $\si^{ab}:=\tau^{-1}g^{ab}\in\Ga(\Herm(T^*\barm)\otimes\Cal E(2))$
  and $L(\si^{ab})\in\Ga(\Cal H^*)$ and the scalar curvature $S$ of
  $\nabla$ admit smooth extensions to all of $\barm$.
\end{cor}
\begin{proof}
  Since the almost c--projective structure determined by $\nabla$ admits a
  smooth extension to the boundary, the same is true for all bundles
  and connections naturally associated to a almost c--projective
  structure. Hence both the tractor connection $\nabla^{\Cal H^*}$ and
  the connection $\nabla^p$ from part (iii) of Proposition
  \ref{prop3.5} are defined and smooth on all of $\barm$. 

  Now by part (iii) of Proposition \ref{prop3.5}, $L(\si^{ab})$ is a
  smooth section of $\Cal H^*$ over $M$, which is parallel for the
  connection $\nabla^p$. Hence it can be smoothly extended to all of
  $\barm$ by parallel transport with respect to $\nabla^p$. Projecting
  this extension to the quotient bundle $\Herm(T^*\barm)\otimes\Cal E(2)$,
  one obtains the required extension of $\si^{ab}$. On the other hand,
  the extension of $L(\si^{ab})$ provides a Hermitean form $H$ on
  $\Cal T^*$ over all of $\barm$. Hence $\det(H)$ is a smooth function
  on $\barm$, and by Proposition \ref{prop4.1} this coincides with $S$
  over $M$, up to a non--zero constant. Hence $S$ admits a smooth
  extension to the boundary.
\end{proof}

\subsection{On extendibility of the almost c--projective
  structure}\label{4.2} In our standard situation $\barm=M\cup\partial
M$, suppose that we have given an almost complex structure $J$ and a
complex linear, minimal connection $\nabla$ on $M$. Then the condition
that the almost c--projective structure determined by $\nabla$ admits
a smooth extension to the boundary can be checked by local
computations.  Consider a complex local frame $\{\xi_1,\dots,\xi_m\}$
for $T\barm$ defined on an open subset $U\subset\barm$. Defining
$\xi_{m+i}=J\xi_i$ for $i=1,\dots, m$, we obtain a real frame
$\xi_1,\dots,\xi_n$ for $T\barm$ over $U$, with respect to which $J$
is represented by a constant matrix. The connection coefficients
$\Ph^i_{jk}$ of $\nabla$ with respect to this frame are defined on
$U\cap M$ by $\nabla_{\xi_j}\xi_k=\sum_i\Ph^i_{jk}\xi_i$. Since the
connection $\nabla$ is complex, they have the property that
$\Ph^i_{j\ell}J^\ell_k=J^i_\ell\Ph^\ell_{jk}$, where $J^i_j$ is the
(constant) coordinate representation of the almost complex structure
$J$. Now one defines the (complex) tracefree part of the connection
coefficients as
$$
\Ps^i_{jk}:=\Ph^i_{jk}-\tfrac1{2m+2}\left(\ph_j\delta^i_k+\ph_k\delta^i_j-
J^\ell_j\ph_\ell J^i_k-J^\ell_k\ph_\ell J^i_j\right),
$$
where $\ph_j:=\Ph^k_{jk}$. One immediately verifies that this
satisfies $\Ps^i_{j\ell}J^\ell_k=J^i_\ell\Ps^\ell_{jk}$ as well as
$\Ps^k_{jk}=0$. By minimality of $\nabla$, the alternation of the
$\Ph^i_{jk}$ is conjugate linear in both arguments and hence
tracefree, which implies that also $\Ps^k_{kj}=0$.

\begin{lemma}\label{lemma4.2}
  The almost c--projective structure determined by $\nabla$ admits a smooth
  extension to $\barm$ if and only if for each boundary point
  $x\in\partial M$, there is a frame $\{\xi_1,\dots,\xi_n\}$ as above,
  defined on an open subset $U\subset\barm$ with $x\in U$, such that
  the tracefree parts $\Ps^i_{jk}$ of the connection coefficients for
  $\nabla$ with respect to the frame, which are initially defined on
  $U\cap M$, admit a smooth extension to $U$.
\end{lemma}
\begin{proof}
  This is completely parallel to the proof of the analogous result for
  projective structures as treated in Proposition 2 of \cite{scalar}.
\end{proof}

By Proposition \ref{prop2.2}, extendibility of the almost c--projective
structure determined by $\nabla$ is a necessary condition for
c--projective compactness of $\nabla$. However, the two conditions are
of different nature, since only very specific connections in a
c--projective class can be c--projectively compact. For example,
Proposition \ref{prop2.3} shows that c--projective compactness for a
connection preserving a volume density implies uniform volume growth
of a specific rate (depending on the dimension). Also, by definition a
c--projectively compact connection cannot admit a smooth extension to
any neighborhood of a boundary point.

\subsection{The case of non--zero scalar curvature}\label{4.3} 
Our second main result is that for an admissible pseudo--Riemannian
metric, extendibility of the almost c--projective structure determined by the
canonical connection $\nabla$ together with a rather weak condition on
the scalar curvature of $\nabla$ implies c--projective compactness.

\begin{thm}\label{thm4.3}
  Suppose that $\barm=M\cup\partial M$, $J$ and $g$ satisfy the
  conditions of Corollary \ref{cor4.1}, so $J$ and $g$ are defined on
  $M$, $g$ is admissible and the almost c--projective structure defined by
  its canonical connection $\nabla$ admits a smooth extension to
  $\barm$. Assume further, that the connection $\nabla$ itself does
  not admit a smooth extension to any open neighborhood of a boundary
  point.

  Suppose that $x\in\partial M$ is such that the smooth extension
  $S:\barm\to\Bbb R$ of the scalar curvature of $\nabla$ provided by
  Corollary \ref{cor4.1} has the property that $S(x)\neq 0$. Then $g$
  is c--projectively compact on a neighborhood of $x$ in $\barm$.
\end{thm}
\begin{proof}
  We denote by $\vol_g$ the volume density of $g$ and put
  $\tau:=(\vol_g)^{-1/(m+1)}\in\Ga(\Cal E(2))$. By Proposition
  \ref{prop2.3}, we can prove the theorem by showing that, locally
  around $x$, $\tau$ can be extended by zero to a defining density for
  $\partial M$.

  By restricting to an appropriate neighborhood of $x$ in $\barm$, we
  may assume that $S$ is nowhere vanishing. Then the section
  $H:=L(\si^{ab})\in\Ga(\Cal H^*)$, where $\si^{ab}=\tau^{-1}g^{ab}$,
  defines a Hermitean form on the bundle $\Cal T^*$ which by
  Proposition \ref{prop4.1} is everywhere non--degenerate. Hence we
  can consider its inverse, which is a non--degenerate Hermitean form
  on $\Cal T$ and therefore defines a smooth section $\Ph\in\Ga(\Cal
  H)$. Over $M$, we can work in the splitting determined by
  $\nabla$. From the proof of Proposition \ref{prop4.1} we know that,
  in this splitting, we have $H=(\tau^{-1}g^{ab},0,C\tau^{-1}S)^t$ for
  some non--zero constant $C$. This shows that in the splitting of
  $\Cal H$ corresponding to $\nabla$, we get
  $\Ph=(C^{-1}S^{-1}\tau,0,\tau g_{ab})^t$, compare with Section
  \ref{3.4}.

  Now the component $C^{-1}S^{-1}\tau$ is the image of $\Ph$ under the
  canonical projection $\Ga(\Cal H)\to\Ga(\Cal E(2))$, so this is
  independent of the choice of splitting. Since $\Ph$ and $S$ are
  smooth up to the boundary, this shows that $\tau$ admits a smooth
  extension to the boundary. Next we claim that this extension
  vanishes along the boundary. Indeed, if $\tau(y)\neq 0$ for some
  $y\in\partial M$, then we can consider a neighborhood $V$ of
  $y\in\barm$ on which $\tau$ is nowhere vanishing. It is then well
  known that there is a unique connection $\hat\nabla$ in the
  restriction of the c--projective class to $V$ for which $\tau$ is
  parallel. But then on $V\cap M$, the section $\tau$ is parallel for
  both $\nabla$ and $\hat\nabla$, so the two connections have to agree
  on $V\cap M$. Thus $\hat\nabla$ provides a smooth extension of
  $\nabla$ to $V$ which contradicts our assumptions.

  Knowing that $\tau$ vanishes along the boundary, it suffices to
  prove that its derivative with respect to any connection which is
  smooth up to the boundary is nowhere vanishing along the
  boundary. To prove this, we have to convert the information on
  $\nabla^{\Cal H^*}_aH$ provided by part (ii) of Proposition
  \ref{prop3.5} into the fact that $\Ph$ satisfies a differential
  equation. Let us first make the definition of $\Ph$ as the inverse of
  $H$ more explicit. As a non--degenerate Hermitean form on $\Cal
  T^*$, the section $H$ gives rise to a isomorphism $A:\Cal T^*\to
  \Cal T$ of vector bundles. Using the (real) dual pairing between
  $\Cal T$ and $\Cal T^*$, this is characterized by $\langle
  A(\xi),\eta\rangle=H(\xi,\eta)$ for $\xi,\eta\in\Cal T^*$. Then one
  uses the inverse $A^{-1}$ to define $\Ph$ via $\Ph(s,t):=\langle
  s,A^{-1}(t)\rangle$ for $s,t\in\Cal T$.

  Since the connection $\nabla^{\Cal H}$ is induced by the standard
  tractor connection on $\Cal T$, also the dual connection
  $\nabla^{\Cal H^*}$ is induced by the tractor connection. Together
  with the above description, this implies that $\nabla_a^{\Cal
    H}\Ph(s,t)=-\nabla_a^{\Cal H^*}H(A^{-1}(s),A^{-1}(t))$. Now we
  need only rather rough information to conclude the argument. Recall
  from \ref{3.2} that the tractor bundle $\Cal T$ contains a natural
  complex line subbundle $\Cal T^1$. Dually, one has the annihilator
  $(\Cal T^1)^\o\subset\Cal T^*$, which is a complex subbundle of
  complex corank one. Now from above we know that the projection of
  $\Ph$ to the quotient bundle $\Cal E(2)$ of $\Cal H$ vanishes along
  the boundary. This implies that, along the boundary, $A^{-1}(\Cal
  T^1)\subset (\Cal T^1)^\o$. On the other hand, part (ii) of
  Proposition \ref{prop3.5} says that $\nabla^{\Cal H^*}_a H$ has
  trivial projection to $\Herm(T^*N)\otimes\Cal E(-2)$, which exactly
  means that $\nabla^{\Cal H^*}_a H$ vanishes (everywhere) upon
  insertion of two sections from $(\Cal T^1)^\o\subset\Cal T^*$. Hence
  we conclude that inserting two sections of $\Cal T^1$ into
  $\nabla_a^{\Cal H}\Ph$, the result vanishes along the boundary.

In terms of splitting into slots, this simply means that the top slot
of $\nabla_a^{\Cal H}\Ph$ (which actually is independent of the choice
of splitting) vanishes along the boundary. Now consider a connection
$\hat\nabla$ in the c--projective class which is smooth up to the
boundary. In the corresponding splitting we must have
$\Ph=(C^{-1}S^{-1}\tau,\ph_a,\ps_{bc})^t$ for some sections $\ph_a$ and
$\ps_{bc}$, which are smooth up to the boundary. Moreover, since
$C^{-1}S^{-1}\tau$ vanishes along the boundary, non--degeneracy of
$\Ph$ implies that $\ph_a$ is nowhere vanishing along the boundary
(since otherwise $\Ph$ would be degenerate). Using formula
\eqref{eq:hconn-def} for $\nabla^{\Cal H}$, we conclude that 
$$
C^{-1}\tau\hat\nabla_a S^{-1}+C^{-1}S^{-1}\hat\nabla_a\tau-2\ph_a
$$ 
vanishes along the boundary. Since $\tau$ vanishes along the boundary,
we conclude that, along the boundary, $\hat\nabla_a\tau=2CS\ph_a$ and
thus is nowhere vanishing. 
\end{proof}

\subsection{Scalar curvature of c--projectively compact
  metrics}\label{4.4} 

For a c--pro\-jec\-tive\-ly compact metric, Corollary \ref{cor4.1}
shows that the scalar curvature of the canonical connection admits a
smooth extension to the boundary. As a first step towards a converse
of Theorem \ref{thm4.3}, we show that the boundary value of this
extension is non--zero on a dense open subset of the boundary.

\begin{prop}\label{prop4.4}
  Let $\barm$ be a smooth manifold of real dimension $2m$ with
  boundary $\partial M$ and interior $M$ and let $J$ be an almost
  complex structure on $M$. Consider a an admissible
  pseudo--Riemannian Hermitean metric $g$ on $(M,J)$, which is
  c--projectively compact. Let $\nabla$ be the canonical connection of
  $g$ and $S:\barm\to\Bbb R$ be the smooth extension of the scalar
  curvature of $\nabla$ guaranteed by Corollary \ref{cor4.1}. Then we
  have.

(1) The set $\{x\in\partial M:S(x)\neq 0\}$ is open and dense in
$\partial M$.

(2) Let $x\in\partial M$ be such that $S(x)\neq 0$, let $\rho$ be a
defining function for $\partial M$ on some neighborhood of $x$ and put
$\th:=-d\rho\o J$. Then locally around $x$, $g$ admits an asymptotic
form
$$
g=-\tfrac{m}2(g^{ij}\Rho_{ij})^{-1}(\tfrac{d\rho^2}{\rho^2}+
\tfrac{\th^2}{\rho^2})+\tfrac{h}{\rho},  
$$ where $\Rho_{ij}$ is the c--projective Schouten tensor of $\nabla$
and $h$ is a Hermitean bilinear form, which is smooth up to the
boundary.

(3) Suppose that $x\in\partial M$ is as in (2) and that
$\xi,\eta\in\frak X(\barm)$ vector fields, which are smooth up to the
boundary. Then locally around $x$, the function $\rho^2g(\xi,\eta)$
admits a smooth extension to the boundary. If for at least one of the
vector fields the insertion into $d\rho$ vanishes along the boundary
and for at least one the insertion into $\th$ vanishes along the
boundary, then even $\rho g(\xi,\eta)$ admits a smooth extension to
the boundary.
\end{prop}
\begin{proof}
As in the proofs of Proposition \ref{prop4.1} and of Theorem
\ref{thm4.3} we consider the tractor $H:=L(\si^{ab})\in\Ga(\Cal H^*)$
determined by the solution $\si^{ab}=\tau^{-1}g^{ab}$ of the metricity
equation determined by $g=g_{ab}$. Knowing that $g$ is c--projectively
compact, we get the additional information that certain specific
connections admit a smooth extension to the boundary. Take a point
$x\in\partial M$ and a local defining function $\rho$ for the boundary
defined on a neighborhood $U$ of $x$ in $\barm$. Consider the
c--projective modification $\hat\nabla:=\nabla+\tfrac{d\rho}{2\rho}$ of
$\nabla$ on $U\cap M$, which by c--projective compactness admits a
smooth extension to all of $U$.

(1) With a view towards contradiction, we assume that $S$ vanishes on
an open subset of $\partial M$ and choose $U$ in such a way that $S$
vanishes on $U\cap\partial M$. From the proof of Proposition
\ref{prop4.1}, we know that on $U\cap M$ and in the scale determined
by $\nabla$, we have $H=(\tau^{-1}g^{ab},0,C\tau^{-1}S)^t$ for some
non--zero constant $C$. Now we can use Proposition \ref{prop3.4} to
compute $H$ on $U\cap M$ in the splitting corresponding to
$\hat\nabla$. By formula \eqref{eq:met-trans}, we obtain
$$
H=\begin{pmatrix} \tau^{-1}g^{ab}
\\ -\frac{1}{\rho}\tau^{-1}\rho_ig^{ic}\\ 
C\tau^{-1}S+\tau^{-1}\tfrac{1}{4\rho^2}\rho_j\rho_kg^{jk}
\end{pmatrix},
$$ where we write $\rho_a$ for $d\rho$. Since $\hat\nabla$ admits a
smooth extension to the boundary, all three slots in this expression
must have the same property. Now recall from Proposition \ref{prop2.3}
that $\tau$ is a defining density for $\partial M$, and thus of the
form $\rho\hat\tau$ for a density $\hat\tau$, which is
nowhere--vanishing along $U\cap\partial M$. Hence $\hat\tau^{-1}$ has
the same property and $\tau^{-1}=\rho^{-1}\hat\tau^{-1}$.

For the top slot we obtain $\hat\tau^{-1}\tfrac{1}{\rho}g^{ab}$. This
implies that for one--forms $\ph=\ph_a$ and $\ps=\ps_b$, which are
smooth up to the boundary, we can write $g^{ij}\ph_i\ps_j$ as $\rho f$
for a function $f$ which admits a smooth extension to the
boundary. For the the middle slot we get
$\hat\tau^{-1}\tfrac{1}{\rho^2}\rho_ig^{ic}$, so
$\tfrac{1}{\rho^2}\rho_ig^{ic}$ admits a smooth extension to the
boundary. This means that for each one--form $\ph=\ph_a$ which is
smooth up to the boundary, $g^{ij}\rho_i\ph_j$ is of the form
$\rho^2f$ for a function $f$ which is smooth up to the boundary. For
the bottom slot, we get
$$
\hat\tau^{-1}\tfrac{1}{4\rho}\left(4CS+
\tfrac{1}{\rho^2}\rho_i\rho_jg^{ij}\right). 
$$
Now both summands in the bracket admit a smooth extension to the
boundary, so we conclude that they have to add up to zero along the
boundary. Assuming that $S$ vanishes along the boundary, we conclude
that $\tfrac{1}{\rho^2}\rho_i\rho_jg^{ij}$ goes to zero along along
the boundary, so $\rho_i\rho_jg^{ij}$ can be written as $\rho^3 f$ for
some function $f$ which is smooth up to the boundary. 

Let us collect the above facts and use them to obtain a description of
the Gram matrix of $g^{ab}$ in terms of an appropriate local frame
consisting of one--forms which are smooth up to the boundary. We use a
complex frame in a real picture, which has $\rho_a$ and $J_a^i\rho_i$
as the first two elements, and then continues with appropriate pairs
of the form $\ph_a$ and $J_a^i\ph_i$. Denoting the resulting
(symmetric) matrix of inner products by $A=(a_{k\ell})$ we know that all
entries vanish along the boundary, the entries in the first two rows
and columns are of the form $\rho^2f_{k\ell}$ for functions
$f_{k\ell}$ which are smooth up to the boundary, while the top left
corner has the form $\rho^3\left(\begin{smallmatrix} a & 0\\ 0 & a
\end{smallmatrix}\right)$ for a function $a$ which is smooth up to the
boundary. Computing $\det(A)$ we can pull out a factor of $\rho^2$
from each of the first two rows and a factor of $\rho$ from each of
the subsequent rows, and then finally one factor of $\rho$ from each
of the first two columns, so $\det(A)=\rho^{2m+4}f$ for some function
$f$ which is smooth up to the boundary. 

On the other hand,
computing the volume density of $g$, we get the product of
$\sqrt{\det(A^{-1})}$ with a nowhere vanishing density which expresses
the change from a local coordinate frame to the frame of the tangent
bundle which is dual to the frame used above. By definition, the
result is $\tau^{-m-1}=\rho^{-m-1}\hat\tau^{-m-1}$. But this implies
that $\rho^{2m+2}\det(A^{-1})$ admits a smooth extension to the
boundary with non--vanishing boundary value $\tilde f$.  Together with
the above, we get
$$
\rho^{2m+2}=\rho^{2m+2}\det(A^{-1})\det(A)=\tilde f\rho^{2m+4}f, 
$$ 
for a function $f$ which is smooth up to the boundary. This leads to
$f=\rho^{-2}\tfrac{1}{\tilde f}$, a contradiction.

\medskip

(2) Knowing that $S(x)\neq 0$, we can choose $U$ in such a way that
$S$ is nowhere vanishing, so also $g^{ij}\Rho_{ij}$ is nowhere
vanishing. As in the proof of Theorem \ref{thm4.3}, we can thus form
the inverse $\Ph\in\Ga(\Cal H)$ of $H$. On $U\cap M$ and in the
splitting corresponding to $\nabla$, we know from that proof that
$\Ph=(2m\tau(g^{ij}\Rho_{ij})^{-1},0,\tau g_{ab})$. Computing the
expression of in the splitting corresponding to the connection
$\hat\nabla=\nabla+\tfrac{d\rho}{2\rho}$ as in part (1), we again know
that all slots in this expression admit smooth extensions to the
boundary. Using formula \eqref{h-trans} we conclude that in the
splitting corresponding to $\hat\nabla$, we get
\begin{equation}\label{hatPhi}
\Ph=\begin{pmatrix} 2m\tau(g^{ij}\Rho_{ij})^{-1} 
\\ m\tfrac{\tau}{\rho}(g^{ij}\Rho_{ij})^{-1}\rho_a \\ 
\tau g_{bc}+m(g^{ij}\Rho_{ij})^{-1}\tau\tfrac{1}{2\rho^2}(\delta^i_b\delta^j_c
+J^i_bJ^j_c)\rho_i\rho_j \end{pmatrix}. 
\end{equation}
The first two slots evidently admit a smooth extension to the
boundary. On $U$, we then define $h_{bc}:=\rho
g_{bc}+\tfrac{m}{2\rho}(g^{ij}\Rho_{ij})^{-1}(\rho_b\rho_c+\th_b\th_c)$. Then
the bottom slot in \eqref{hatPhi} equals $\hat\tau h_{bc}$, so
$h_{bc}$ admits a smooth extension to the boundary. But this exactly
means that we get the required asymptotic form.

\medskip

(3) From part (2) we conclude that 
$$
\rho^2 g(\xi,\eta)=-\tfrac{m}2(g^{ij}\Rho_{ij})^{-1}
(d\rho(\xi)d\rho(\eta)+\th(\xi)\th(\eta))+\rho h(\xi,\eta). 
$$ 
The right hand side evidently admits a smooth extension to the
boundary. Under the additional assumptions on $\xi$ and $\eta$, we can
write one of the factors $d\rho$ and one of the factors $\th$ as
$\rho$ times a function which admits a smooth extension to the
boundary, which implies the second claim.
\end{proof}

\subsection{Necessity of the asymptotic form}\label{4.5} 
The last step to obtain converses to Theorems \ref{thm2.6} and
\ref{thm4.3} is showing that the scalar curvature of a c--projectively
compact metric is asymptotically (locally) constant. 

\begin{thm}\label{thm4.5}
Let $\barm$ be a smooth manifold with boundary $\partial M$ and
interior $M$. Let $J$ be an almost complex structure on $M$ and let
$g$ be an admissible Hermitean pseudo--Riemannian metric on $(M,J)$
which is c--projectively compact. Suppose that the resulting extension
of $J$ has the property that $\partial M$ is non--degenerate and that
the Nijenhuis tensor has asymptotically tangential values. 

Then the boundary value of the smooth extension of the scalar
curvature $S$ of the canonical connection $\nabla$ of $g$ guaranteed
by Corollary \ref{cor4.1} is locally constant, and $g$ admits an
asymptotic form as in formula \eqref{asymp-form} in Section \ref{2.6}
satisfying the conditions stated there.
\end{thm}
\begin{proof}
By Proposition \ref{prop4.4} the boundary value of $S$ is
non--vanishing on an open dense subset of $\partial M$ and we work
locally around a point $x$ in this subset. Let $\nabla$ be the
canonical connection of $g$ and for a local defining function $\rho$
for the boundary let ${}^\rho\nabla$ be the corresponding
c--projective modification of $\nabla$ which admits a smooth extension
to the boundary. Choose a vector field $\mu$ along the boundary such
that $d\rho(\mu)$ is identically one on the boundary whereas
$\th(\mu)$ vanishes along the boundary. As in the proof of Lemma 1 of
\cite{proj-comp2}, we can extend $\mu$ to a vector field defined
locally around $x$ such that ${}^\rho\nabla_\mu\mu=0$. Moreover, we
can use the flow of $\mu$ to identify an open neighborhood of $x$ in
$\barm$ with $V\x [0,\ep)$ for an open neighborhood $V$ of $x$ in
  $\partial M$ and some $\ep>0$.

Now consider the function $\rho^2g(\mu,\mu)$ which admits a smooth
extension to the boundary by part (3) of Proposition \ref{prop4.4}. By
part (2) of that proposition, this is given by  
$$
\rho^2g(\mu,\mu)=\hat CS^{-1}(d\rho(\mu)^2+\th(\mu)^2)+\rho h(\mu,\mu)
$$ 
for some constant $\hat C$. Along $\partial M$, we have $d\rho(\mu)=1$
and $\th(\mu)=0$, so this approaches $\hat CS^{-1}$ at the boundary. 

Now consider a vector field $\underline{\xi}\in\frak X(\partial M)$,
which is a section of the CR--subbundle. Then we can extend this to a
vector field $\ze$ on a neighborhood of $x$ such that $d\rho(\ze)=0$
and $\th(\ze)=0$. The derivative of the boundary value of $\hat
CS^{-1}$ in direction $\underline{\ze}$ can then be computed as the
boundary value of $\ze\cdot(\rho^2g(\mu,\mu))$. 

\textbf{Claim 1}: For any vector field $\ze$ such that $d\rho(\ze)=0$
and $\th(\ze)=0$, the function $\ze\cdot (\rho^2g(\mu,\mu))$ vanishes
along the boundary. 

To prove this claim, we first observe that $\ze\cdot\rho=0$, so on
$M$, we can write $\ze\cdot(\rho^2g(\mu,\mu))$ as
$2\rho^2g(\nabla_\ze\mu,\mu)$. Since $g$ is admissible, the torsion of
$\nabla$ equals $-\frac14\Cal N$, where $\Cal N$ is the Nijenhuis
tensor, so the definition of torsion shows that 
\begin{equation}\label{nximu}
\nabla_\ze\mu=\nabla_\mu\ze+[\ze,\mu]-\tfrac14\Cal N(\ze,\mu). 
\end{equation}
By assumption on the Nijenhuis tensor, $d\rho(\Cal N(\ze,\mu))$ and
$\th(\Cal N(\ze,\mu))$ both vanish along the boundary. Hence part (3)
of Proposition \ref{prop4.4} shows that $\rho^2g(-\frac14\Cal
N(\ze,\mu),\mu)$ vanishes along the boundary. 

Likewise, expanding $0=dd\rho(\ze,\mu)$ we obtain
$d\rho([\ze,\mu])=-\mu\cdot d\rho(\ze)+\ze\cdot d\rho(\mu)$. By
construction, $d\rho(\ze)=0$ and $d\rho(\mu)=1+\rho f$ for some smooth
function $f$, so since $\ze\cdot\rho=0$, we see that
$d\rho([\ze,\mu])$ vanishes along the boundary. Since also $\th(\mu)$
vanishes along the boundary, we can again use part (3) of Proposition
\ref{prop4.4} to conclude that $\rho^2g([\ze,\mu],\mu)$ vanishes along
the boundary.

Finally, over $M$, we can write 
\begin{equation}\label{gnmuxi}\begin{aligned}
\rho^2g(\nabla_\mu\ze,\mu)&=\rho^2\mu\cdot
g(\ze,\mu)-\rho^2g(\ze,\nabla_\mu\mu)\\
&=\rho\mu\cdot(\rho g(\ze,\mu))-\rho
d\rho(\mu)g(\ze,\mu)-\rho^2g(\ze,\nabla_\mu\mu). 
  \end{aligned}
\end{equation}
Now the construction of $\mu$ together with formula \eqref{cpdef} for
a c--projective modification implies that
\begin{equation}\label{nmumu}
\nabla_\mu\mu=-\frac{d\rho(\mu)}{\rho}\mu+\frac{d\rho(J\mu)}{\rho}J\mu= 
-\frac{d\rho(\mu)}{\rho}\mu-\frac{\th(\mu)}{\rho}J\mu. 
\end{equation}
Inserting this into \eqref{gnmuxi}, we see that
$\rho^2g(\nabla_\mu\ze,\mu)$ can be written as 
$$
 \rho\mu\cdot(\rho g(\ze,\mu))-\th(\mu)\rho g(\ze,J\mu). 
$$
Now by part (3) of Proposition \ref{prop4.4}, both $\rho g(\ze,\mu)$
and $\rho g(\ze,J\mu)$ admit smooth extensions to the boundary, so
since $\th(\mu)$ vanishes along the boundary, we see that also
$\rho^2g(\nabla_\mu\ze,\mu)$ vanishes along the boundary. This
completes the proof of Claim 1. 

\medskip

Hence we have verified that the derivative of the boundary value of
$\hat CS^{-1}$ in any direction lying in the CR subspace of $T\partial
M$ vanishes. By non--degeneracy, this implies that this boundary value
is locally constant. Denoting by $C$ this locally constant function on
the boundary extended constantly along flow lines of $\mu$, we see
that $\hat CS^{-1}=C+\rho f$ for a function $f$ which is smooth up to
the boundary. Hence part (2) of Proposition \ref{prop4.4} shows that
$g$ admits an asymptotic form as in formula \eqref{asymp-form} from
Section \ref{2.6}. To complete the proof of the theorem, it thus
remains to verify that for vector fields $\xi$ and $\ze$ such that
$d\rho(\ze)$ and $\th(\ze)$ vanish, $h(\xi,J\ze)$ approaches
$Cd\th(\xi,\ze)$ at the boundary.

\textbf{Claim 2}: Suppose that $\xi$ is a vector field that is smooth
up to the boundary such that $d\rho(\xi)$ vanishes along the
boundary. Then, for $\ze$ as above, $Cd\rho({}^\rho\nabla_\xi\ze)$
approaches $\frac12h(\ze,\xi)$ at the boundary.

To prove this claim, we first observe that by part (3) of Proposition
\ref{prop4.4} the function $\xi\cdot(\rho g(\ze,\mu))$ admits a
smooth extension to the boundary, and, on $M$ we can compute this as 
$$
d\rho(\xi)g(\ze,\mu)+\rho g(\nabla_\xi\ze,\mu)+\rho
g(\ze,\nabla_\xi\mu). 
$$ 
Since $d\rho(\xi)$ vanishes along the boundary, the first term in this
sum admits a smooth extension to the boundary. Next, we compute 
$$
\nabla_\xi\ze={}^\rho\nabla_\xi\ze-\frac{d\rho(\xi)}{2\rho}\ze-
\frac{\th(\xi)}{2\rho}J\ze. 
$$
The coefficient of $\ze$ in the second summand on the right hand side
admits a smooth extension to the boundary, so hooking this into $\rho
g(\_,\mu)$ one obtains a function which admits a smooth extension to
the boundary. Likewise, we compute
$$
\nabla_\xi\mu={}^\rho\nabla_\xi\mu-\frac{d\rho(\xi)}{2\rho}\mu-
\frac{d\rho(\mu)}{2\rho}\xi-\frac{\th(\xi)}{2\rho}J\mu-
\frac{\th(\mu)}{2\rho}J\xi. 
$$
Here the first, second, and the last summand in the right hand side
admit smooth extensions to the boundary and hence hooking them into
$\rho g(\ze,\_)$ one obtains functions which admit a smooth extension
to the boundary. The upshot of this is that 
\begin{align*}
\rho g({}^\rho\nabla_\xi\ze,\mu)&-\tfrac12\th(\xi)g(J\ze,\mu)-
\tfrac12d\rho(\mu)g(\ze,\xi)-\tfrac12\th(\xi)g(\ze,J\mu)\\
&=\rho g({}^\rho\nabla_\xi\ze,\mu)-\tfrac12d\rho(\mu)g(\ze,\xi)
\end{align*}
admits a smooth extension to the boundary. Hence if we multiply by
$\rho$ we must get a function tending to $0$ at the boundary. But
since we already know that we have the asymptotic form for $g$, we can
evaluate this directly which readily leads to the statement in Claim
2.

\medskip

This already suffices to the compute the boundary value of
$Cd\th(\xi,\ze)$ in the case that also $d\rho(\xi)$ and $\th(\xi)$
vanish identically. In this case,
$Cd\th(\xi,\ze)=-C\th([\xi,\ze])=Cd\rho(J[\xi,\ze])$. Since the
torsion of ${}^\rho\nabla$ has values in the CR subspace along the
boundary, we can compute the boundary value of this as the boundary
value of 
$$
Cd\rho(J({}^\rho\nabla_\xi\ze-{}^\rho\nabla_\ze\xi))=Cd\rho
({}^\rho\nabla_\xi J\ze)-Cd\rho ({}^\rho\nabla_\ze J\xi)
$$
and use claim 2 to compute both summands in the right hand side. This
leads to $\tfrac12h(J\ze,\xi)-\tfrac12 h(J\xi,\ze)=h(\xi,J\ze)$. 

To complete the proof, it suffices to show that we get the right
boundary value for $\xi=J\mu$. Then $\th(\xi)=d\rho(\mu)=1+\rho f$ for
some smooth function $f$. Since $d\rho(\ze)=0$ we conclude that
$\ze\cdot\th(\xi)$ goes to zero at the boundary, so we can again
compute $Cd\th(\xi,\ze)$ as 
$$ 
Cd\rho({}^\rho\nabla_\xi J\ze)-Cd\rho ({}^\rho\nabla_\ze J\xi). 
$$ 
The first summand can be computed using claim 2, so to complete the
proof, it suffices to show that $Cd\rho({}^\rho\nabla_\ze\mu)$
approaches $\frac12h(\ze,\mu)$ at the boundary. But in the proof of
claim 1, we have already seen that $\rho^2g(\nabla_\ze\mu,\mu)$
approaches zero at the boundary, so $\rho g(\nabla_\ze\mu,\mu)$ admits
a smooth extension to the boundary. Rewriting
$$
\nabla_\ze\mu={}^\rho\nabla_\ze\mu-\frac{d\rho(\mu)}{2\rho}\ze-
\frac{\th(\mu)}{2\rho}J\ze 
$$
and observing that the coefficient of $J\ze$ in the last summand
admits a smooth extension to the boundary, this follows in the same
way as claim 2. 
\end{proof}

\subsection{Curvature asymptotics}\label{4.6}
We next analyze the curvature of c--projectively compact special
affine connections and, more specifically, the canonical connections
of c--projectively compact metrics. As a first step, we describe the
asymptotic behavior of the projective Schouten tensor. As we have
noted in \ref{3.1}, the c--projective Schouten tensor $\Rho_{ab}$ is
not symmetric in general. Let us denote by $\be_{ab}:=\Rho_{[ab]}$ its
skew symmetric part. Further, the symmetric part $\Rho_{(ab)}$ is a
symmetric bilinear form on the tangent spaces, which are complex
vector spaces. Hence this can be decomposed into a Hermitean part
$\Rho^+_{ab}:= \frac12(\Rho_{(ab)}+J_a^iJ^j_b\Rho_{(ij)})$ and an
  anti--Hermitean part $\Rho^-_{ab}:=
  (\frac12\Rho_{(ab)}-J_a^iJ^j_b\Rho_{(ij)})$.   
\begin{lemma}\label{lem4.6}
Let $\barm$ be a smooth manifold with boundary $\partial M$ and
interior $M$. Let $J$ be an almost complex structure on $M$ and let
$\nabla$ be a minimal complex linear connection on $TM$ which is
c--projectively compact. Suppose further that the smooth extension of
$J$ to the boundary has the property that its Nijenhuis tensor has
asymptotically tangential values.

Let $\rho$ be a local defining function for $\partial M$, put
$\rho_a=d\rho$ and $\th_a=-J^i_a\rho_i$, and let $\hat\nabla$ be
the corresponding c--projective modification of $\nabla$. Then for the
c--projective Schouten tensor $\Rho_{ab}$ of $\nabla$, 
$$
\rho\Rho_{ab}+\tfrac1{4\rho}(\rho_a\rho_b+\th_a\th_b)
$$ 
admits a smooth extension to the boundary and its boundary value
coincides with $\tfrac12\hat\nabla_a\rho_b$. In particular, $\be_{ab}$
and $\rho\Rho^-_{ab}$ admit smooth extensions to the boundary.   
\end{lemma}
\begin{proof}
Denoting by $\hat\Rho_{ab}$ the c--projective Schouten tensor of
$\hat\nabla$, formula \eqref{Rhotransf} for the change of Schouten
tensor shows that 
$$
\Rho_{ab}=\hat\Rho_{ab}+\hat\nabla_a\Up_b+\Up_a\Up_b-J_a^iJ_b^j\Up_i\Up_j, 
$$
where $\Up_a=\tfrac1{2\rho}\rho_a$. Now
$\hat\nabla_a\Up_b=-\frac1{2\rho^2}\rho_a\rho_b+
\tfrac{1}{2\rho}\hat\nabla_a\rho_b$ and adding $\Up_a\Up_b$, we get
$-\frac1{4\rho^2}\rho_a\rho_b+\tfrac{1}{2\rho}\hat\nabla_a\rho_b$. On
the other hand, the last summand in the right hand side gives
$-\frac1{4\rho^2}\th_a\th_b$. Bringing two terms to the other side and
multiplying by $\rho$, we get
\begin{equation}\label{Rhoasymp}
\rho\Rho_{ab}+\tfrac1{4\rho}(\rho_a\rho_b+\th_a\th_b)=
\tfrac12\hat\nabla_a\rho_b+\rho\hat\Rho_{ab}. 
\end{equation}
Since $\hat\Rho_{ab}$ is smooth up to the boundary, it only remains to
prove the claims on $\be_{ab}$ and $\Rho^-_{ab}$. 

Skew symmetrizing in \eqref{Rhoasymp} over $a$ and $b$ and using that
$\hat\nabla_a\rho_b$ is symmetric along the boundary by Lemma
\ref{lem2.2a}, we obtain $\rho\be_{ab}=\rho\hat\be_{ab}$, so
$\be_{ab}$ admits a smooth extension to the boundary.

On the other hand, symmetrizing over $a$ and $b$ in \eqref{Rhoasymp},
does not affect the second summand in the left hand side, which in
addition is Hermitean. Hence forming the symmetric anti--Hermitean
part in the left hand side, we just get $\rho\Rho^-_{ab}$, whereas in
the right hand side all terms still admit a smooth extension to the
boundary.
\end{proof}

Suppose that $\ph_{ab}$ is a tensor field on an almost complex
manifold which is symmetric and Hermitean, i.e.~such that
$\ph_{ab}=\ph_{(ab)}=J^i_aJ^j_b\ph_{ij}$. Then one can associate to
$\ph$ a tensor field $C_{ab}{}^c{}_d$ defined by 
$$
C_{ab}{}^c{}_d:=2(\de^c_{[a}\ph_{b]d}-J^i_{[a}\ph_{b]i}J^c_d-J^c_{[a}\ph_{b]i}J^i_d).  
$$ This is obviously skew symmetric in $a$ and $b$ and a direct
computation shows that $C_{[ab}{}^c{}_{d]}=0$, so $C$ has the
symmetries of a curvature tensor. Direct computations also show that
$C_{ab}{}^i{}_dJ^c_i=C_{ab}{}^c{}_jJ^j_d$, so $C$ has values in
complex linear maps. Finally, one shows that
$C_{ab}{}^c{}_d=J^i_aJ^j_bC_{ij}{}^c{}_d$, so $C$ is of type
$(1,1)$. Hence we call $C$ the \textit{complex rank one curvature
  tensor associated to $\ph$}. Using this, we can now describe the
curvature asymptotics of a general c--projectively compact
connection. 

\begin{prop}\label{prop4.6}
Let $\barm$ be a smooth manifold with boundary $\partial M$ and
interior $M$. Let $J$ be an almost complex structure on $M$ and let
$\nabla$ be a minimal complex linear connection on $TM$ which is
c--projectively compact. Suppose further that the resulting smooth
extension of $J$ to the boundary has the property that its Nijenhuis
tensor has asymptotically tangential values. 

Then the curvature tensor $R$ of $\nabla$ has the property that for
any local defining function $\rho$ for the boundary, $\rho^2 R$ admits
a smooth extension to the boundary with boundary value
$-\tfrac{1}{4}C_{ab}{}^c{}_d$, where $C_{ab}{}^c{}_d$ is the complex
rank one curvature tensor ,  associated to
$\rho_a\rho_b+\th_a\th_b$.

If $J$ is integrable, then the boundary value of $\rho
R_{ab}{}^c{}_d+\tfrac1{4\rho}C_{ab}{}^c{}_d$ is given by 
$$
\de^c_{[a}\hat\nabla_{b]}\rho_d-J^i_{[a}(\hat\nabla_{b]}\rho_i)J^c_d-
J^c_{[a}(\hat\nabla_{b]}\rho_i)J^i_d
$$
\end{prop}
\begin{proof}
The characterization of the c--projective Schouten tensor
$\Rho_{ab}$ in part (1) of Proposition \ref{prop3.1} can be
reformulated as  
\begin{equation}\label{Rdecomp}
R_{ab}{}^c{}_d=W_{ab}{}^c{}_d+2\big(\de^c_{[a}\Rho_{b]d}-\Rho_{[ab]}\de^c_d-
J^i_{[a}\Rho_{b]i}J^c_d-J^c_{[a}\Rho_{b]i}J^i_d\big), 
\end{equation}
where $W_{ib}{}^i{}_d=0$. Now let $\hat\nabla={}^\rho\nabla$ be the
c--projective modification of $\nabla$ determined by $\rho$. Then of
course the associated Weyl curvature $\hat W_{ab}{}^c{}_d$ is smooth
up to the boundary. Now the relation between $W_{ab}{}^c{}_d$ and
$\hat W_{ab}{}^c{}_d$ is described in Proposition 2.13 of
\cite{CEMN}. This is done in a complexified picture, discussing the
$(p,q)$--components of the two--form $W$ with values in $\End_{\Bbb
  C}(T\bar M)$. The components of $(p,q)$--types $(2,0)$ and $(1,1)$
of $W_{ab}{}^c{}_d$ and $\hat W_{ab}{}^c{}_d$ agree. In type $(0,2)$,
the difference of the components of $\hat W$ and $W$ is obtained by
contracting the torsion of $\nabla$ into a tensor $v_{ab}{}^c$ which 
is a linear combination of tensor products of one factor $\Up_\ell$
with either a Kronecker delta or two copies of $J$. Since
$\rho\Up_\ell=\tfrac12\rho_\ell$ is smooth up to the boundary, we see
that $\rho v_{ab}{}^c$ admits a smooth extension to the boundary, so
$\rho W_{ab}{}^c{}_d$ admits a smooth extension to the boundary. 

Multiplying \eqref{Rdecomp} by $\rho^2$, we thus conclude form Lemma
\ref{lem4.6} that the right hand side admits a smooth extension to the
boundary. Moreover, the first summand and the term involving
$\Rho_{[ab]}$ do not contribute to the boundary value, and in the
other summands we may replace $\Rho_{ab}$ by
$-\tfrac14(\rho_a\rho_b+\th_a\th_b)$ without changing the boundary
value. Hence the first result follows from the definition of
$C_{ab}{}^c{}_d$.

If $J$ is integrable, then by Proposition 2.13 of \cite{CEMN}, the
$(0,2)$--component of $W_{ab}{}^c{}_d$ vanishes for any connection in
the c--projective class. Hence from above we conclude that in this
case $W_{ab}{}^c{}_d$ admits a smooth extension to the boundary.  So
up to terms admitting a smooth extension to the boundary, we can write
$R_{ab}{}^c{}_d$ as in \eqref{Rdecomp} but leaving out the terms
containing $W_{ab}{}^c{}_d$ or $\Rho_{[ab]}$ in the right hand
side. Then the claimed expression follows directly from Lemma
\ref{lem4.6}.
\end{proof}

\subsection{The asymptotic Einstein property}\label{4.7} 
We conclude the article by discussing the curvature asymptotics for
c--projectively compact metrics. In particular, we show that the
canonical connection associated to such a metric automatically
satisfies an asymptotic version of the Einstein equation.

\begin{thm}\label{thm4.7}
  Let $\barm$ be a smooth manifold with boundary $\partial M$ and
  interior $M$. Let $J$ be an almost complex structure on $M$ and let
  $g$ be an admissible Hermitean metric on $M$, which is
  c--projectively compact. Suppose further that the Nijenhuis tensor
  of the resulting smooth extension of $J$ to the boundary has
  asymptotically tangential values. Let $\Rho_{ab}$ be the Schouten
  tensor of the canonical connection $\nabla$ of $g$ and $\Rho^+_{ab}$
  its symmetric Hermitean part.

(1) The tracefree part
$\Rho^\o_{ab}:=\Rho^+_{ab}-\tfrac1{2m}g^{ij}\Rho_{ij}g_{ab}$ of
$\Rho^+$ admits a smooth extension to the boundary.

(2) If $J$ is integrable, then $\Rho^+_{ab}=\Rho_{ab}$, so the
K\"ahler metric $g$ satisfies an asymptotic version of the Einstein
equation. In this case, let $\rho$ be a local defining function for
the boundary, put $\rho_a=d\rho$ and $\th_a=-J_a^i\rho_i$, let
$(d\th)_{ab}$ be the exterior derivative of $\th_a$ and let
$C_{ab}{}^c{}_d$ be the complex rank--one curvature tensor associated
to $\rho_a\rho_b+\th_a\th_b$. Then up to terms admitting a smooth
extension to the boundary, the curvature $R_{ab}{}^c{}_d$ of $g$ is
given by
$$
-\tfrac{1}{4\rho^2}C_{ab}{}^c{}_d+\tfrac1{2\rho}\left(
\delta^c_{[a}J^i_{b]}(d\th)_{di}-(d\th)_{ab}J^c_d+
J^c_{[a}J^i_{b]}(d\th)_{ij}J^j_d\right). 
$$
\end{thm}
\begin{proof}
(1) Let $\rho$ be a local defining function for the boundary. Taking
  the Hermitean part in formula \eqref{Rhoasymp} from the proof of
  Lemma \ref{lem4.6}, we conclude that
  $\rho\Rho^+_{ab}+\tfrac1{4\rho}(\rho_a\rho_b+\th_a\th_b)$ admits a
  smooth extension to the boundary with boundary value the symmetric
  Hermitean part of $\tfrac12\hat\nabla_a\rho_b$. By Lemma
  \ref{lem2.2a}, $\hat\nabla_a\rho_b$ is symmetric along the boundary,
  so we just have to take the Hermitean part
  $\tfrac14(\hat\nabla_a\rho_b+J^i_aJ^j_b\hat\nabla_i\rho_j)$.

  On the other hand, we use the asymptotic form for $g$ from formula
  \eqref{asymp-form} from Section \ref{2.6} provided by Theorem
  \ref{thm4.5}. Comparing to part (2) of Proposition \ref{prop4.4}, we
  see that that constant $C$ occurring in this asymptotic form equals
  the boundary value of $-\tfrac{m}2(g^{ij}\Rho_{ij})^{-1}$. Hence we
  conclude that
  $$
\rho\frac{g^{ij}\Rho_{ij}}{2m}g_{ab}+\tfrac1{4\rho}(\rho_a\rho_b+\th_a\th_b)
$$
admits a smooth extension to the boundary with boundary value
$\tfrac{-1}{4C}h_{ab}$, where $h_{ab}$ is the
Hermitean form that occurs in the asymptotic form of $g$.   Hence we
conclude that $\rho\Rho^\o_{ab}$ admits a smooth extension to the
boundary with boundary value 
$$
\tfrac14\left(\hat\nabla_a\rho_b+J^i_aJ^j_b\hat\nabla_i\rho_j+
\tfrac{1}{C}h_{ab}\right). 
$$
Now let us contract $J$ into the $b$--index of the expression in the
bracket. Using symmetry and the fact that $J$ is parallel, the
boundary values of the result coincides with the one of 
\begin{equation}
  \label{ntheta}
 \hat\nabla_aJ_b^i\rho_i-\hat\nabla_bJ^i_a\rho_i+\tfrac{1}{C}h_{ai}J^i_b=
\tfrac{1}{C}h_{ai}J^i_b-\hat\nabla_a\th_b+\hat\nabla_b\th_a.   
\end{equation}
Now since the Nijenhuis tensor also hooks trivially into $\th_a$, the
boundary values of the last two terms give $(d\th)_{ab}$, so
\eqref{ntheta} vanishes by Theorem \ref{thm4.5}. Hence we see that the
smooth extension of $\rho\Rho^\o_{ab}$ to the boundary vanishes along
the boundary, so $\Rho^\o_{ab}$ itself admits a smooth extension to
the boundary.

(2) If $J$ is integrable, then by Proposition \ref{prop2.5}, $g$ is
(pseudo--)K\"ahler metric, so its Ricci curvature is well known to be
Hermitean, which also implies that $\Rho_{ab}$ is Hermitean. By the
proof of part (1), we conclude that in this case $\hat\nabla_a\rho_b$
is Hermitean along the boundary and coincides with $\tfrac12
J^i_a(d\th)_{bi}$ along the boundary. Using this, the form of the
curvature follows readily from Proposition \ref{prop4.6}.
\end{proof}

\subsection{The case of asymptotically parallel Nijenhuis
  tensor}\label{4.8} In the case of a non--integrable almost complex
structure, all the results of Section \ref{4} are based on the
canonical connection $\nabla$ of a quasi--K\"ahler metric $g$ and not
on its Levi--Civita connection $\nabla^g$. In particular, all
conditions and results on curvature concern the curvature of $\nabla$.
To conclude this article, we study an asymptotic vanishing condition
on the covariant derivative of the Nijenhuis tensor, which in
particular is satisfied for nearly K\"ahler metrics. Assuming this
condition, we show that in the results of Section \ref{4}, one may
often use the curvature of $\nabla^g$ instead of the curvature of
$\nabla$.

\begin{definition}\label{def4.8}
Let $\barm$ be a smooth manifold with boundary $\partial M$ and
interior $M$. Let $J$ be an almost complex structure on $M$ with
Nijenhuis tensor $\Cal N$, and let $\nabla$ be a linear connection on
$TM$ which preserves $J$. We say that $\Cal N$ is asymptotically
parallel for $\nabla$ if the tensor field $\nabla\Cal N$ admits a
smooth extension to all of $\barm$ which vanishes along the boundary.  
\end{definition}

\begin{prop}\label{prop4.8}
Let $\barm$ be a smooth manifold with boundary $\partial M$ and
interior $M$. Let $J$ be an almost complex structure on $\barm$ with
Nijenhuis tensor $\Cal N$, and let $g$ an admissible Hermitean metric
on $(M,J)$ such that
\begin{itemize}
\item $\Cal N$ as asymptotically tangential values
\item $\Cal N$ is asymptotically parallel for the canonical
connection $\nabla$ of $g$.
\end{itemize}
Then for the curvature $R$ of $\nabla$ and the curvature $R^g$ of the
Levi--Civita connection $\nabla^g$ of $g$, the difference $R-R^g$
admits a smooth extension to all of $\barm$.
\end{prop}
\begin{proof}
Consider the contorsion tensor $A$ defined by
$\nabla^g_\xi\eta=\nabla_\xi\eta-A(\xi,\eta)$. The explicit form of
$A$ has already been used in the proof of Proposition
\ref{prop2.5}. Replacing $J\zeta$ by $\zeta$ in that formula and
contracting with an inverse metric, the result reads in abstract index
notation as
$$
A^c_{ab}=\tfrac12\left(-g_{ib}T^i_{aj}g^{jc}-g_{ia}T^i_{bj}g^{jc}+T_{ab}^c\right), 
$$
where $T=T^c_{ab}$ is the torsion of $\nabla$ and hence equals
$-\tfrac14\Cal N$. Since $\Cal N$ has asymptotically tangential
values, part (3) of Proposition \ref{prop4.4} implies that $\rho
g_{ib}T^i_{aj}$ admits a smooth extension to the boundary. From the
proof of that Proposition, we also know that $\rho^{-1}g^{jc}$ admits
a smooth extension to to the boundary. Hence we conclude that the
contorsion tensor $A$ admits a smooth extension to the
boundary. Computing the covariant derivative $\nabla A$ using that
$\nabla$ is metric, we see that we only get covariant derivatives
hitting $T$, so we similarly conclude that $\nabla A$ admits a smooth
extension to the boundary. But then from the definition of curvature,
one easily concludes that for $\xi,\eta,\ze\in\frak X(M)$, one may
write the difference $R^g(\xi,\eta)(\ze)-R(\xi,\eta)(\ze)$ as
$$
-(\nabla A)(\xi,\eta,\ze)+(\nabla
A)(\eta,\xi,\ze)-A(T(\xi,\eta),\ze)+A(\xi,A(\eta,\ze))-A(\eta,A(\xi,\ze)), 
$$
which implies the result. 
\end{proof}

\begin{cor}\label{cor4.8} 
Under the assumptions of Proposition \ref{prop4.8}, which are in
particular satisfied if $g$ is a nearly K\"ahler metric (of any
signature), one may replace curvature quantities associated to the
canonical connection $\nabla$ by quantities associated to the
Levi--Civita connection $\nabla^g$ of $g$ as follows:
\begin{itemize}
\item scalar curvature in Theorem \ref{thm4.5} 
\item Schouten tensor in Lemma \ref{lem4.6} and Theorem \ref{thm4.7}
\item full curvature tensor in Proposition \ref{prop4.6}.
\end{itemize}
\end{cor}
\begin{proof}
  First note that on a nearly K\"ahler manifold, the Nijenhuis tensor
  is globally parallel for the canonical connection, see Remark 4.5 in
  \cite{CEMN} and \cite{Nagy}, so the assumptions of Proposition
  \ref{prop4.8} are satisfied in this case. Next, in any case where the assumptions of 
 Proposition
  \ref{prop4.8} are satisfied we have that $R$ and $R^g$ differ by terms which admit a
  smooth extension to the boundary. Forming the Ricci contraction, we
  conclude that the same holds for the difference of the Ricci tensors
  and hence for the difference of the Schouten tensors of the two
  connections. Finally, the scalar curvature is obtained by
  contracting an inverse metric $g^{ab}$ into the Ricci curvature. As
  we have noted in the proof of Proposition \ref{prop4.8},
  $\rho^{-1}g^{ab}$ admits a smooth extension to the boundary. Thus we
  conclude that the difference of the scalar curvatures of the two
  connections not only admits a smooth extension to the boundary but
  this extension also vanishes along the boundary.  From this all the
  claims follow from the statements of the results listed in the
  corollary.
\end{proof}

\end{document}